\documentclass[11pt]{article}
\usepackage{amsmath,amssymb,color,amsthm}
\usepackage{authblk}
\usepackage{graphicx}
\usepackage{epsfig}
\usepackage{fullpage}
\usepackage{cite}
\usepackage{enumerate}
\numberwithin{equation}{section}
\usepackage{tikz}\usetikzlibrary{arrows}

\newcommand{\R}{\mathbb{R}}

\newcommand{\T}{\mathbb{T}}

\newtheorem{thm}{Theorem}[section]
\newtheorem{lem}[thm]{Lemma}
\newtheorem{prop}[thm]{Proposition}
\newtheorem{corr}[thm]{Corollary}
\newtheorem{define}[thm]{Definition}

\newtheorem{remark}[thm]{Remark}
\newtheorem{obs}[thm]{Observation}

\newcommand{\1}{\mathbf{1}}

\newcommand{\av}[1]{\left|{#1}\right|}
\newcommand{\ip}[2]{\left\langle{{#1}},{{#2}}\right\rangle}
\newcommand{\norm}[1]{\left\|{#1}\right\|}
\renewcommand{\l}{\left}

\renewcommand{\r}{\right}
\newcommand{\mb}[1]{\mathbf{#1}}
\DeclareMathOperator{\sign}{sign}
\DeclareMathOperator{\proj}{proj}

\newcommand{\J}{{\bf J}}
\newcommand{\D}{{\bf D}}
\newcommand{\A}{{\bf A}}
\renewcommand{\phi}{\varphi}
\newcommand{\bomega}{{\boldsymbol{\omega}}}
\newcommand{\btheta}{{\boldsymbol{\theta}}}

\newcommand{\uu}{{\mathbf{u}}}
\newcommand{\vv}{{\mathbf{v}}}
\newcommand{\w}{{\mathbf{w}}}
\newcommand{\z}{{\mathbf{z}}}

\begin{document}

\title{Configurational stability for the Kuramoto--Sakaguchi model}
\author[1]{Jared Bronski}
\author[2]{Thomas Carty}
\author[1]{Lee DeVille}
\affil[1]{Department of Mathematics, University of Illinois}
\affil[2]{Department of Mathematics, Bradley University}

\date{\today}

\maketitle

\abstract{The Kuramoto--Sakaguchi model is a modification of the well-known Kuramoto model that adds a phase-lag paramater, or ``frustration'' to a network of phase-coupled oscillators.  The Kuramoto model is a flow of gradient type, but adding a phase-lag breaks the gradient structure, significantly complicating the analysis of the model. We present several results determining
the stability of phase-locked configurations:  the first of these gives a sufficient condition for stability, and the second a sufficient condition for instability. (In fact, the instability criterion gives a count, modulo 2, of the dimension of the unstable manifold to a fixed point and having an odd count is a sufficient condition for instability of the fixed point.)  We also present numerical results for both small and large collections of Kuramoto--Sakaguchi oscillators.}



{\bf

The Kuramoto-Sakaguchi model is a fundamental model for the study of phase-locking phenomena of coupled oscillators where a natural frustration or de-tuning parameter is inherent in the underlying system.  Since its introduction in 1987, this model has been used extensively to model, among other things, chemical oscillation, neural networks, and laser arrays.  However the vast majority of the analysis has been numerical.  The reason for this is that the addition of a frustration parameter to a network of phase-coupled oscillators causes the system to lose much of the natural symmetries associated with the standard Kuramoto model.  Over the last decade, mathematical results for the Kuramoto-Sakaguchi model mostly have focused on the behavior of the system in the mean-field limit - i.e. the number of oscillators goes to infinity.  It is only very recently that people have returned to rigourously analyzing finite network Kuramoto-Sakaguchi systems analytically.   In this current work we present two results, a sufficient condition for stability and a method for counting the number of eigenvalues with positive real part modulo two, which gives a sufficient condition for instability.
}

\section{Introduction}

\subsection{Problem Formulation}

We consider the differential equation on $\T^n$:
\begin{equation}\label{eq:KS}
\frac{d\theta_i}{dt} = \omega_i + \gamma \sum_{j=1}^n \left(\sin(\theta_j-\theta_i-\alpha)+\sin(\alpha)\right).
\end{equation}
where $i=1,\dots,n$, and the parameters $\omega_i, \gamma\in \R$ and
$\av\alpha<\pi/2$.  This system was originally analyzed in a series of
papers~\cite{Kuramoto.Sakaguchi.1986, Sakaguchi.etal.1987,
  Sakaguchi.etal.1988} and is commonly known today\footnote{A
  rereading of the early literature suggests that a more fitting name
  for this system of equations would be the
  Sakaguchi--Shinomono--Kuramoto (SSK) system of equations, but
  against the weight of a consensus in the literature the gods
  themselves contend in vain.} as the Kuramoto--Sakaguchi
system~\cite{DeSmet.Aeyels.2007, Omelchenko.Wolfrum.2013,
  Amadori.Ha.Park.17, HaKimPark2017remarks, HaKoZhang2018emergence}.  This system is a generalization of the
well-studied Kuramoto system, which is obtained by setting $\alpha$ to
zero in~\eqref{eq:KS}. (In what follows, we will often refer to
the system with $\alpha=0$ as the ``standard Kuramoto'' system.) The
parameter $\alpha$ is alternatively called the {\em phase-lag}, {\em
  detuning}, or {\em frustration} parameter,
see~\cite{Kiss.Zhai.Hudson.2005} for physical justification of each of
these terms in the context of chemical oscillations.

This current work extends the spectral analysis performed upon the standard Kuramoto model originally pioneered in~\cite{MS1} and reconsidered in~\cite{Bro.DeV.Park.2012}.  From a mathematical point of view, the addition of the nonzero $\alpha$ parameter leads to a significant increase in difficulty in the analysis of~\eqref{eq:KS} as compared to standard Kuramoto.  In~\cite{Bro.DeV.Park.2012}, when $\alpha=0$, the system~\eqref{eq:KS} is a gradient flow, and in particular the Jacobian at any fixed point is symmetric, simplifying the analysis considerably.

The addition of $\alpha$ also adds a level of dynamical complexity to this model.  In standard Kuramoto,  the center of mass of the system~\eqref{eq:KS} rotates around the circle at a rate given by the average of the $\omega_i$; for fixed $\bomega$ any two solutions will precess at the same rate.  In contrast, the system~\eqref{eq:KS} can support multiple configurations that precess at different rates for the same $\bomega$.  This is one feature that is in stark contrast to the standard Kuramoto model, and requires a rethinking of many of the intuitions associated with that model.

\subsection{Phase-locking and projections}\label{subsec:Stabilty is different in KS}

\newcommand{\f}{{\mathbf{f}}}

Generally we will find it useful to denote the vector field $\f\colon\T^n\rightarrow \R^n$ defined as
\begin{equation}\label{eq:defoff}
f_i(\btheta,\alpha)=\sum_{j=1}^n \left(\sin(\theta_j-\theta_i-\alpha)+\sin(\alpha)\right),
\end{equation}
and we can write~\eqref{eq:KS} compactly as
\begin{equation}\label{eq:KSvec}
\frac{d\btheta}{dt}=\bomega+\gamma\f(\btheta,\alpha).
\end{equation}
It is clear from this formulation that scaling $\gamma$ is equivalent to scaling $\bomega$, so in this paper we choose the convention throughout that $\gamma = 1$.

We first note that our definition in~\eqref{eq:KS} is slightly different than that commonly chosen in most studies, where there is no $\sin(\alpha)$ term.  Of course, this only shifts the vector field by a constant amount and has no effect on the Jacobian of the system, but it has the nice normalization that $\btheta=0$ is a fixed point for $\bomega=0$ and any $\alpha$.  In particular, notice that the function $\sin(\cdot-\alpha) + \sin\alpha$ has a fixed point at 0 with positive derivative whenever $\av \alpha < \pi/2$ ---   this makes it ``most like'' standard Kuramoto.  In particular, it follows directly that if we choose $\bomega=0$ and any $\av\alpha<\pi/2$, then $\btheta = 0$ is an attracting fixed point.
Moreover, if we consider the family of functions $\sin(\cdot-\alpha) + \sin\alpha$ for $-\pi/2<\alpha<\pi/2$, then the effect is that the  ``stable'' point at zero remains fixed, while one of the unstable points moves toward the origin, and at $\alpha=\pm\pi/2$ there is a saddle-node bifurcation.

The fundamental question considered in this paper is  whether~\eqref{eq:KS} (or~\eqref{eq:KSvec}) admits a  phase-locked solution and whether or not this solution is dynamically stable.

\begin{define}\label{def:phaselocked-dyanamicstabliity}
We say that a solution to~\eqref{eq:KS} is {\bf  phase-locked} if $\btheta(t)$ is a solution and if $\theta_i(t)-\theta_j(t)$ is constant for every $i,j$.  Equivalently, $\btheta(t)$ is a phase-locked solution if $\btheta(t) = \btheta_0 + c t \1$; thus any phase-locked solution can thus either be a fixed configuration, or a rigidly rotating configuration.  In this case, we say the phase-locked solution rotates at velocity $c$. By {\bf dynamically stable} we mean that perturbations of a solution decay back to it, i.e. if $\btheta(t)$ is a dynamically stable configuration, then there is an $\epsilon>0$ such that for all $v$ with $\norm v < \epsilon$, if $\btheta(0) = \widehat\btheta(0) + v$, then
\begin{equation*}
  \lim_{t\to\infty} \norm{\proj_{\mathbf{1}^{\perp}}(\btheta(t) - \widehat\btheta(t))} = 0,
\end{equation*}
 where $\proj_{\mathbf{1}^{\perp}}$ denotes the orthogonal projection
 onto the orthogonal complement of the vector $\mathbf{1}^\perp.$ We say that ``$\bomega$ gives rise to a phase-locked solution'' if there is a phase-locked solution for~\eqref{eq:KSvec} with that $\bomega$.
 \end{define}

Note that $\f(\cdot,\alpha)$ has an $U(1)$ symmetry:
\begin{equation*}
  \f(\btheta_0 + c{\bf 1},\alpha) = \f(\btheta_0,\alpha),
\end{equation*}
which necessitates the $\proj_{\mathbf{1}^{\perp}}$ in the definition
above. Given any $\btheta_0\in\R^n$, we can choose $\bomega$ to make this a fixed point of our system, by choosing $\bomega_0 = - \f(\btheta_0)$.  Then $\btheta_0$ is a fixed point for~\eqref{eq:KSvec} with $\bomega=\bomega_0$.
Let $\btheta(t)$ be a solution of~\eqref{eq:KS} for some $\bomega_0$.  Choose $\boldsymbol{\eta}(t) = \btheta(t) + \kappa t {\bf 1}$.
\begin{align*}
  \boldsymbol{\eta}'(t) &= \btheta'(t) + \kappa{\bf 1}\\
  	  &= \bomega_0 + \f(\btheta(t),\alpha)+ \kappa{\bf 1}\\
	  &= (\bomega_0 + \kappa {\bf 1}) + \f(\eta(t),\alpha),
\end{align*}
where in the last line we exploited the $U(1)$ symmetry of $\f$.  This means that shifting a solution with a constant velocity is equivalent to shifting the frequencies by a constant, and vice versa.

From this it follows that if $\bomega$ gives rise to a phase-locked solution, then $\bomega+ c{\bf 1}$ does as well.  Thus let us define:

\newcommand{\roa}{ R_\bomega^{(\alpha)}}
\newcommand{\soa}{ L_\bomega^{(\alpha)}}
\newcommand{\poa}{ P_\bomega^{(\alpha)}}

\begin{define}
  For fixed $\alpha$, we define the {\bf frequency \{region, slice, projection\}} as the sets $\roa,\soa,\poa$ where
 \begin{align*}
   \roa &= \{\omega:\mbox{~\eqref{eq:KSvec} has a phase-locked solution}\},\\
   \soa &= \{\omega:\mbox{~\eqref{eq:KSvec} has a fixed point}\},\\
   \poa &= \mbox{orthogonal projection of $\roa$ onto $\mathbf{1}^\perp$}.
 \end{align*}
\end{define}

For standard Kuramoto, the distinction between frequency slice and projection is not important as $L_\omega^{(0)} = P_\omega^{(0)}$.  The fact that $\soa\neq \poa$ for general $\alpha$ might be surprising to those used to standard Kuramoto.  For standard Kuramoto, fixed points, whether they be stable or unstable, always precess according to their average frequency.  Additionally, for any $\bomega$ we can translate via $\boldsymbol{\eta}(t)$ with $\kappa=\frac{1}{n}\sum \omega_i$ to shift to an equivalent system with the new $\bomega$ lying in the mean zero plane $\mathbf{1}^\perp$.  Then, for standard Kuramoto, $\mathbf{f}$ maps the mean zero plane in the configuration space $\btheta$ into the mean zero plane in the frequency space $\bomega$.  However, this construction does not work for general $\alpha$.

We demonstrate this for $N=3$ oscillators in Figure~\ref{fig:ShadowGraphs} which depicts the sets
$\soa$ and $\roa$ for $\alpha = 0,\frac{\pi}{12},\frac{\pi}{6},\frac{\pi}{3}$. The coordinates are
arranged as follows: the vector $\left(\frac{1}{\sqrt{3}}, \frac{1}{\sqrt{3}},
\frac{1}{\sqrt{3}}\right)^t$ is oriented in the $\hat z$ direction, while the graphs are drawn over
the (mean zero) configuration space: $\btheta := \hat{x}\left(\frac1{\sqrt{2}},-\frac{1}{\sqrt{2}},0\right) + \hat{y}\left(\frac{1}{\sqrt{6}},\frac{1}{\sqrt{6}},-\frac{2}{\sqrt{6}}\right)$.  When $\alpha=0$ the set $\soa$ is two dimensional and agrees with the set $\roa$. (Note: for visual clarity the $\roa$ region has been shifted by one unit in the negative $\hat z$
direction. ) As $\alpha$ is increased $\soa$ becomes increasingly non-planar. Also note the loss of symmetry: when $\alpha=0$ the region has symmetry group $D_6$, but for $\alpha\neq 0$ the symmetry group is $D_3. $ When $\alpha =0$ the standard Kuramoto is invariant under permutations of the oscillators along with $\btheta \mapsto -\btheta, \bomega \mapsto-\bomega$. This symmetry is lost for non-zero $\alpha$.
\begin{figure}[h]
\begin{center}
\includegraphics[width=0.4\textwidth]{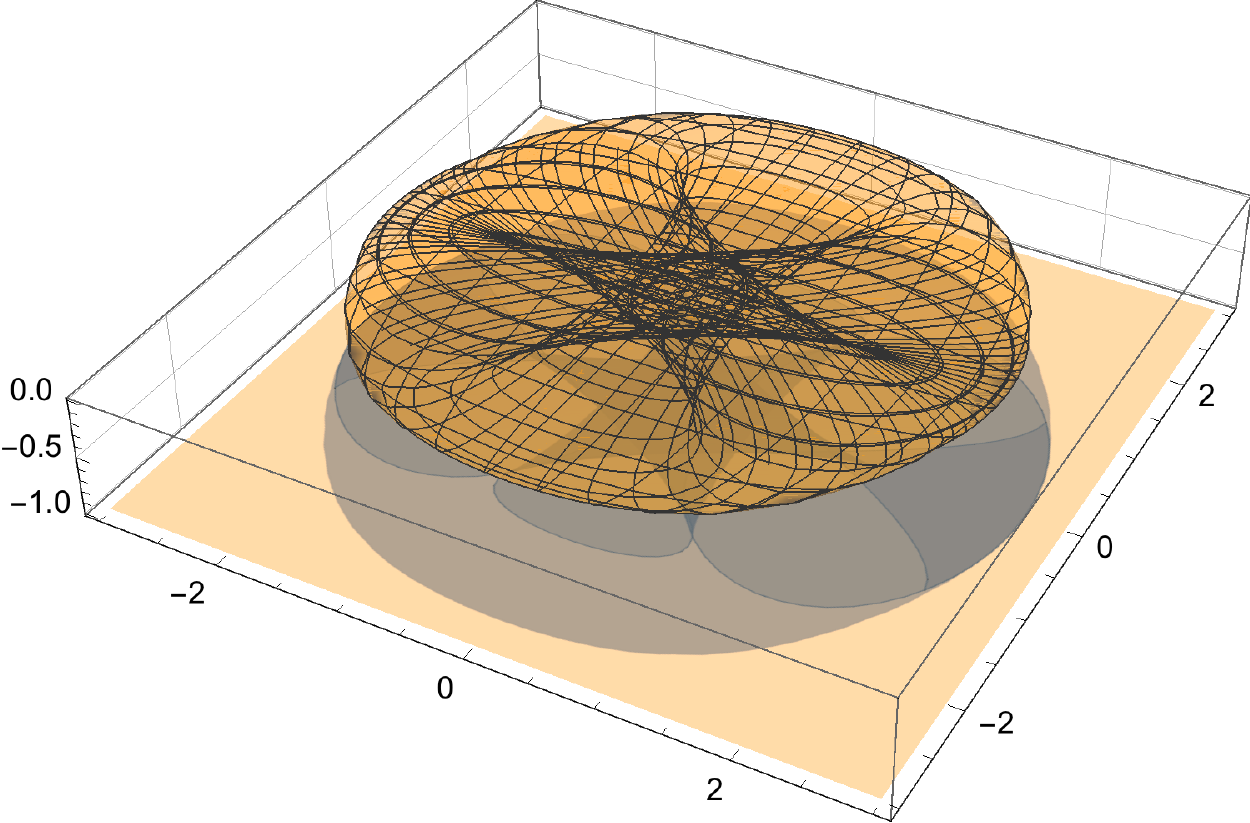}\includegraphics[width=0.4\textwidth]{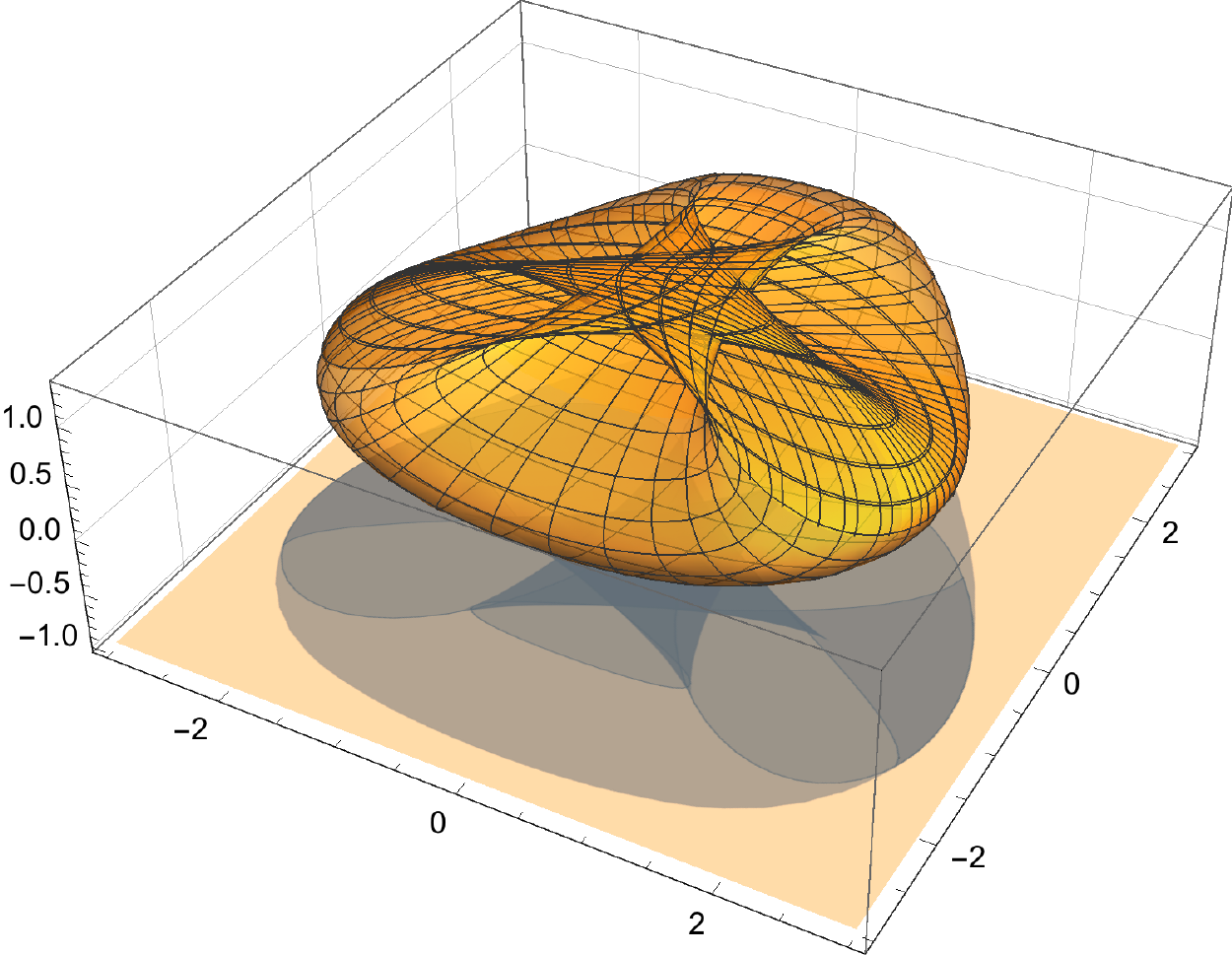}\\
\includegraphics[width=0.4\textwidth]{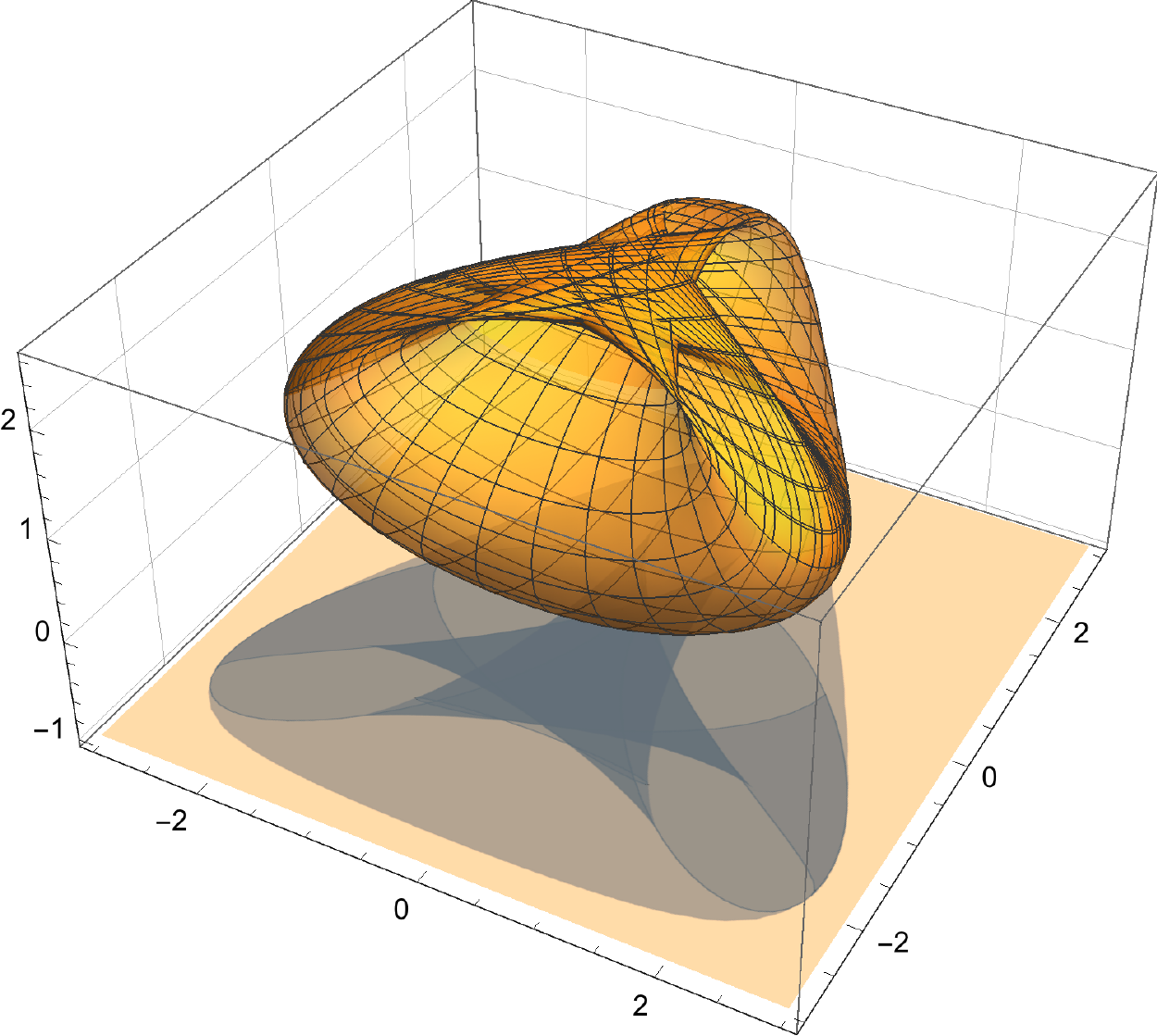}\includegraphics[width=0.4\textwidth]{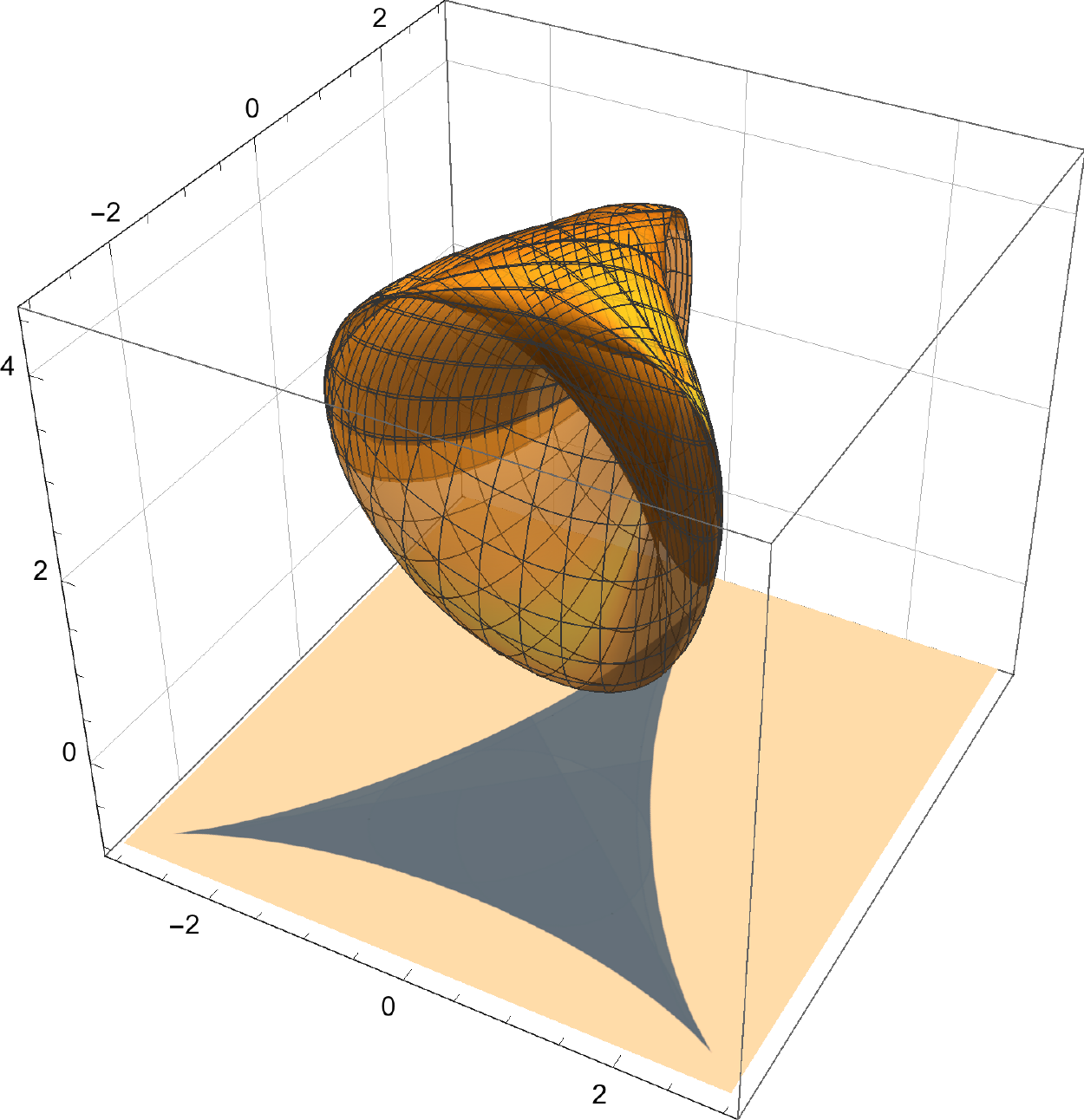}
\end{center}
\caption{$N=3$, the regions  $\soa$ and $\poa$ for  $\alpha \in \{0,\frac{\pi}{12},\frac{\pi}{6},\frac{\pi}{3}\}$ }
\label{fig:ShadowGraphs}
\end{figure}

The lifting of $\soa$ out of the mean zero plane has a very interesting impact on solutions.  For $\alpha\neq 0$, there are choices of pairs $\btheta,\bomega$ such that $\btheta$ is a fixed point (and thus has average velocity zero) while the $\bomega$ has nonzero average!  Again using an specific $\boldsymbol{\eta}$--shift, we can also see that this implies that there are cases with $\sum\omega_i = 0$, but~\eqref{eq:KS} supports precessing phase-locked solutions.  Continuing in this direction, the system~\eqref{eq:KS} can support solutions that precess at different velocities for even the same $\bomega$.   We exhibit this in the case $n=3$.  Modding for translations and reflections, for standard Kuramoto there are six fixed points (with zero angular velocity):
\begin{equation*}
(0,0,0),\quad (2\pi/3,0,4\pi/3), \quad (0,2\pi/3,4\pi/3),\quad (0,\pi,0), \quad (0,0,\pi), \quad (\pi,0,0).
\end{equation*}
(Note: for clarity, we have not projected these fixed-points into the
mean zero plane.)  For Kuramoto--Sakaguchi with $\bomega\neq 0$, there
are still six phase-locked solutions.  However these different
phase-locked solutions rotate with different angular frequencies.  To
see this, first note that $\f((0,0,0),\alpha) = (0,0,0)$, so the
solution with all angles the same is a fixed point for $\bomega=0$.
If we next consider the twist state  $(0,2\pi/3,4\pi/3)$, we can
compute $\f((0,2\pi/3,4\pi/3),\alpha) = 3\sin\alpha(1,1,1)$.  Thus
this solution rotates with angular velocity $3\sin\alpha$ or
(equivalently) is a fixed point for  $\bomega = -3\sin\alpha(1,1,1)$.   Thus we see that $\bomega =0$ supports both the fixed point $\btheta=0$ and the phase-locked equilateral configuration which rotates at velocity $3\sin\alpha$.

For the three fixed points that are permutations of a single $\pi$,
the situation is slightly more complicated, in that the equilibrium
configuration depends on the parameter $\alpha$.  For example, to consider solutions near to $(0,\pi,0)$, we write $\btheta(\alpha) = (0,\pi+\phi(\alpha),0)$. Then we see that
\begin{equation*}
  \f(\btheta(\alpha)) = (\sin (\alpha -\phi (\alpha ))+\sin (\alpha ),2 (\sin (\phi (\alpha )+\alpha )+\sin (\alpha )),\sin (\alpha -\phi (\alpha ))+\sin (\alpha )).
\end{equation*}
If these components are all equal, we obtain a phase-locked solution.  This condition is the functional equation
\begin{equation}
  0 = \sin (\alpha -\phi (\alpha ))-\sin (\alpha )-2 \sin (\alpha+\phi (\alpha ) ),
\end{equation}
which is equivalent to
\[ \cos(\phi(\alpha)) \sin(\alpha) + 3\cos(\alpha)\sin(\phi(\alpha)) =-\sin(\alpha).
\]
Applying some trigonometric identities this can be solved to find that
\[
\phi(\alpha) =  - 2\arctan \left(\frac{\sin\alpha}{3 \cos\alpha}\right).
\]
As before, $(0,\pi+\phi(\alpha),0)$ is a fixed point for the system with $\bomega=-(\sin(\alpha)+\sin(\alpha-\phi(\alpha)))(1,1,1)$.  Equivalently, we see that $\bomega=0$ support the solution $\btheta(t) = (0,\pi+\phi(\alpha),0)-(\sin (\alpha -\phi (\alpha ))+\sin (\alpha ))(1,1,1)$, which rotates at velocity $\sin(\alpha)+\sin(\alpha-\phi(\alpha))$.

\begin{prop}
  The set $\roa$ can be generated by translation by either $\soa$ or $\poa$, and we can represent $\soa$ as a graph over $\poa$, in the sense that there is a function $z\colon\poa\to\R$ such that if $y$ is any coordinate system in ${\bf 1}^\perp$, then $\poa = (y,0), z\in\poa$ and $\soa = (y,z(y)), z\in \poa$.
\end{prop}

\begin{proof}
  For generation:  Since $\poa$ is the projection of $\roa$ this is clear.  The claim about $\soa$ will follow from the second claim.

  To see the second claim, let $\bomega\in \soa$.  This means that there is a $\btheta$ with $F(\btheta) = -\bomega$.  Note that for any $c\neq 0$, the system~\eqref{eq:KSvec} has a phase-locked solution with velocity $c$, and therefore $\bomega$ is the unique intersection of $\soa$ with the line $\bomega+c{\bf 1}$.
\end{proof}

\subsection{Low rank analysis of the Jacobian}

In the spirit of the work in~\cite{Bro.DeV.Park.2012}, we begin by recasting the Jacobian of the forcing function of~\eqref{eq:KS} as a low-rank perturbation of a diagonal matrix.

\begin{lem}
\label{linearized flow form lemma}
The linearized flow of~\eqref{eq:KS} takes the form
\[
\frac{d{\bf x}}{dt} = {\bf J} {\bf x}
\]
where ${\bf J}$ is a non-symmetric Laplacian matrix. Moreover ${\bf J}$ can be decomposed in the form
\[
{\bf J} = {\bf D} + {\bf A}
\]
where $D$ is diagonal and ${\bf A}$ of rank at most $2$.
\end{lem}

\begin{proof}
A straightforward calculation gives the following expression for the Jacobian matrix
\begin{align}
\label{J}
\mathbf{J}  &:= [J_{ij}] \text{ where } J_{ij}=[\partial_{\theta_j}\dot{\theta_i}] =\left\{\begin{array}{ll}
-\sum_{i\neq j}\cos{(\theta_j-\theta_i-\alpha)}, & i=j\\
\cos{(\theta_j-\theta_i-\alpha)}, & i\neq j
\end{array}\right.
\end{align}
Equivalently,
\[
J_{ij}=\left\{\begin{array}{ll}
-\sum_{j}\cos{(\theta_j-\theta_i-\alpha)}+\cos{(\theta_i-\theta_i-\alpha)}, & i=j\\
\cos{(\theta_j-\theta_i-\alpha)}, & i\neq j
\end{array}\right.
\]
Then $\mathbf{J}$ can be decomposed as the sum of the diagonal matrix $\mathbf{D}$,
\[
\mathbf{D}=[D_{ij}] \text{ where } d_{ij}=-\delta_{ij} \sum_{j}\cos{(\theta_j-\theta_i-\alpha)},
\]
and the  matrix $\mathbf{A}$,
\begin{equation}\label{eq:defofA}
\mathbf{A}=[a_{ij}] \text{ where } a_{ij}=\cos{(\theta_j-\theta_i-\alpha)}.
\end{equation}

To show that $\mathbf{A}$ is at most rank two, we will show that it can be written as the sum of two rank one matrices.  Note that
\begin{align*}
    a_{ij} &= \cos{(\theta_j-\theta_i-\alpha)}\\
            &= \cos{(\theta_j-\alpha/2-(\theta_i+\alpha/2))}\\
            &= \cos{(\theta_j-\alpha/2)}\cos{(\theta_i+\alpha/2)}+\sin{(\theta_j-\alpha/2)}\sin{(\theta_i+\alpha/2)}.
\end{align*}
Thus
\[
\mathbf{A}=[\cos{(\theta_i+\alpha/2)}\cos{(\theta_j-\alpha/2)}]+[\sin{(\theta_i+\alpha/2)}\sin{(\theta_j-\alpha/2)}].
\]
Each row of the first matrix in the decomposition of $\mathbf{A}$ is a scalar multiple of the vector $[\cos{(\theta_j-\alpha/2)}]^t$.  Hence it is rank one.  Similarly, each row in the second matrix is a scalar multiple of the vector $[\sin{(\theta_j-\alpha/2)}]^t$.  Thus, $\mathbf{A}$ is at most a rank two matrix.
\end{proof}

The fact that the Jacobian matrix of the flow can be written as the sum of a diagonal piece
and a low-rank (rank two) piece will simplify the spectral analysis.


\begin{obs}
\label{J as rank two perturb}
The matrix $\J$ can be written in the form
\[
\J = \D + \uu\otimes\vv + \w\otimes\z.
\]
More specifically $\A=\uu\otimes\vv + \w\otimes\z$ where
\begin{align*}
\uu= [\cos{(\theta_i+\alpha/2)}]^t\\
\vv= [\cos{(\theta_j-\alpha/2)}]^t\\
\w= [\sin{(\theta_i+\alpha/2)}]^t\\
\z= [\sin{(\theta_j-\alpha/2)}]^t
\end{align*}
and $\D$ is diagonal with diagonal entries
\[
D_{ii} = -\sum_{j}\cos{(\theta_j-\theta_i-\alpha)}.
\]
\end{obs}

Additionally, examining the original statement of $\J$ \eqref{J} we can easily see that for any $\boldsymbol{\theta}$ all the row sums of the matrix $\J$ are always zero.  This leads to our second observation.

\begin{obs} It is always true that
$\mathbf{J1}= \mathbf{0}$, and thus zero is always an eigenvalue of $\mathbf{J}$.  Thus any stable fixed point of~\eqref{eq:KS} is only semi-stable, in that it has a ``soft mode'' that arises from the translation invariance.  In a different context, it was exactly this semi-stability quality of any fixed point that required the use of the projection onto the mean-zero plane in defining dynamically stable (Definition~\ref{def:phaselocked-dyanamicstabliity}).
\end{obs}


The function $\mathbf{f}:\T^n\rightarrow \R^n$ is a natural map from the configuration space $\T^n$ to the frequency space $\R^n$.  Since~\eqref{eq:KS} is invariant under the one-parameter family of rotations $\theta_i\rightarrow \theta_i+s$, we can restrict our work in $\T^n$ to the mean-zero plane, $\sum \theta_i=0$.  In other words, we need only consider $\boldsymbol{\theta}$ an element of the reduced configuration space $\mathcal{T}:=\T^n\bigcap\mathbf{1}^{\perp}$.
\begin{define}
\label{def_fully_synch_soln}
We define $\mathcal{S}_{\theta}$ to be the set of configurations in $\mathcal{T}$ for which the Jacobian $\J=\partial \mathbf{f}/\partial \boldsymbol{\theta}$ \eqref{J} is negative semi-definite with a one dimensional kernel. We define $\mathcal{S}_{\omega}$ to be the set of frequencies in $\R^n$ given by image of $\mathcal{S}_{\theta}$ under $\mathbf{f}$.  That is, $\mathcal{S}_{\omega}:=\mathbf{f}(\mathcal{S}_{\theta})$.\end{define}

 As demonstrated in Figure~\ref{fig:ShadowGraphs}, the range of the map $\mathbf{f}$ is not all of $\R^n$. However, provided we can show that the set $\mathcal{S}_{\theta}$ is non-empty, then the image of the reduced configuration space will be an $n-1$ dimensional surface in $\R^n$.  It is clear that $\mathcal{S}_{\omega}$ is the important object for studying synchronization: All questions about the probability of full synchrony are questions about the size of $\mathcal{S}_{\omega}$ in some measure.  One of the key ingredients in this is a good characterization of $\mathcal{S}_{\omega}$. In standard Kuramoto, it is common at this point to reduce the frequency space of $\boldsymbol{\omega}$ to the mean-zero plane as well, but this is not possible as discussed above.  Nonetheless, $\mathbf{f}$ remains a well-defined map for the Kuramoto-Sakaguchi model between $\mathcal{S}_{\theta}$ and $\mathcal{S}_{\omega}$.  Colloquially, $\mathcal{S}_{\omega}$ can be thought of as a graph over $\mathcal{S}_{\theta}$ and in the orientation of Figure~\ref{fig:ShadowGraphs} we will see that $\mathcal{S}_{\omega}$ is the bottom surface of the object.

Finally, note that $f(\mathbf{0})=\mathbf{0}$, the derivative of $\mathbf{f}$ is the Jacobian $\mathbf{J}$, and the dimension of the kernel of $\mathbf{J}$ is one at the origin, so at least in a neighborhood of the origin, $S_\omega$ is a manifold of the same dimension as $S_\theta$.

%
%


\section{Characterization of the Stable Set}

Recall that for any $\boldsymbol{\theta}$, the Jacobian matrix $\J$ \eqref{J} has the property that the row sums are always zero.  Hence $0$ is always an eigenvalue of $\J$.  Determining the stable and unstable regions of $\T^n$ is a matter of determining the real part of the rest of the eigenvalues associated with $\J$.  By Observation~\ref{J as rank two perturb}, we know that $\J$ can be decomposed into $\J = \D + \uu\otimes\vv + \w\otimes\z$ where the diagonal entries of $\D$ are of the form $D_{ii} = -\sum_{j}\cos{(\theta_j-\theta_i-\alpha)}$.  In \cite{Bro.DeV.Park.2012}, the eigenvalue analysis was accomplished using a homotopy argument in the spirit of the Birman--Schwinger Principle.  Consider the one parameter family of operators
\begin{align}
\label{J of s operator}
\J(s):=\D+s(\uu\otimes\vv + \w\otimes\z).
\end{align}
Clearly $\J(0)=\D$ and $\J(1)=\J$.  For the standard Kuramoto model, the drift of the eigenvalues was able to be accurately detected via changes in the size of the kernel of $\J(s)$ as $s$ increased from 0 to 1.  The key component of that homotopy argument relied on the fact that, when $\alpha=0$, $\A$ is a positive definite self-adjoint matrix.  This fails for the Kuramoto-Sakaguchi model, since the matrix $\A= \uu\otimes\vv + \w\otimes\z$ is no longer self-adjoint in general.

In Section~\ref{stability result} we overcome the asymmetry of $\J$ by a different approach using Perron--Frobenius.  In Section~\ref{index-thm-instability}, we return to the homotopy argument to derive an index theorem that gives a more complete description of the eigenvalue drift for any $\boldsymbol{\theta}$ in the reduced configuration space $\mathcal{T}:=\T^n\bigcap\mathbf{1}^{\perp}$.  In the end, this yields a nice characterization for much of the unstable region in $\mathcal{T}$.

\subsection{A Stability Result} \label{stability result}

For any $|\alpha|<\pi/2$, there is a neighborhood of $\mathbf{0}$ such that all the diagonal entries $D_{ii}$ are negative.  Since $\mathbf{J}$ is not symmetric, as $s$ increases from zero to one in~\eqref{J of s operator}, the eigenvalues need not be monotone increasing.  However, for a particular subset of $\mathcal{T}$, we can control the top eigenvalue.  If we can then show that the top eigenvalue is 0, then we will necessarily have a $\theta$ that must be in the set of $\mathcal{S}_{\theta}$. To do so, we need a couple of corollaries to the Perron--Frobenius Theorem~\cite[Section 8.4]{Horn.Johnson.book}.  For completeness, we will state the relevant portion of the theorem:
\begin{thm}[Perron--Frobenius]
  Let $M$ be a matrix with positive entries.  Then the spectral radius $\rho(M)$ is a simple eigenvalue of $M$.  The left  and right eigenvectors with eigenvalue $\rho(M)$ have components of the same sign and thus without loss of generality can be chosen to to have positive entries.  Moreover,
  \begin{equation*}
  \min_i \sum_j m_{ij} \le \rho(M) \le \max_i \sum_j m_{ij}.
  \end{equation*}
\end{thm}

Although Perron--Frobenius is typically stated for a matrix with all positive entries, the really important mechanism is the positivity of the off-diagonal entries.  Specifically, we have:
\begin{corr}\label{cor:PF}
\begin{enumerate}
\item Let $M$ be a zero row sum matrix with positive off-diagonal entries.  Then zero is a simple eigenvalue of $M$, all other eigenvalues have negative real parts, and the vectors in the left and right nullspace have all positive entries.

\item Let $M$ be a matrix with positive off-diagonal entries and negative row sums.  Then the eigenvalue of $M$ with largest real part is itself real, negative, and no larger than the largest row sum. Its associated eigenvectors have all positive entries.
\end{enumerate}
\end{corr}

\begin{proof}
  Let $M$ be a matrix with positive off-diagonal entries.  Define $B=M+cI$, where $c$ chosen large enough to make all of the entries of $B$ positive.  If $M$ has zero row sums, then all of the row sums of $B$ are $c$, and therefore $\rho(B) = c$ is a simple eigenvalue whose associate eigenvectors can be chosen to have positive entries.  Thus $0$ is a simple eigenvalue of $M$, and the associated eigenvectors can be chosen to have positive entries.  Since $M=B-cI$, these eigenvectors are also eigenvectors of $M$.

  Similarly, if $M$ has negative row sums, then all of the row sums of $B$ are strictly less than $c$, and therefore the top eigenvalue of $M$ is negative.
\end{proof}

In order to use these results on $\J(s)$, we will need to be slightly more careful about the region in $\mathcal{T}$ that we use.  Recall that $\A=[\cos(\theta_j-\theta_i-\alpha)]$ \eqref{eq:defofA}.  To use the corollary, we need to restrict $\boldsymbol{\theta}$ further.
\begin{define}
\label{Lees stable region}
We define $\mathcal{S_{\theta}^{\dag}}$ to be the set of configurations in $\mathcal{T}$ such that $\cos(\theta_i - \theta_j - \alpha) > 0$ for all $i,j$.
\end{define}
Note that for any element of $\mathcal{S_{\theta}^{\dag}}$, the diagonal of $\D$ will have all negative entries.

\begin{thm}[$\mathcal{S_{\theta}}$ is non-empty]
\label{J eta is stable on Lees region}
For any $\boldsymbol{\theta}\in \mathcal{S_{\theta}^{\dag}}$, the matrix $\J(s)$ is stable for all $s\in[0,1]$ and unstable for $s>1$.  Thus $\mathcal{S_{\theta}^{\dag}}\subseteq\mathcal{S_{\theta}}$, the set of fully synchronous solutions to the Kuramoto-Sakaguchi model.  In particular, for $\av\alpha<\pi/2$, the set $\mathcal{S_{\theta}^{\dag}}$ and thus $\mathcal{S_{\theta}}$ is nonempty.
\end{thm}
\begin{proof}
For each $s\in[0,1]$, the matrix $D+s A$ satisfies the assumptions of Corollary~\ref{cor:PF}, and we can write the left and right eigenvectors as $y(s), x(s)$ (we always choose the normalization that $\ip{y(s)}{x(s)} = 1$).  Note that for $s\in[0,1)$, the row sums of $A$ are strictly negative, and for $s=1$ they are zero.

Let us denote $\lambda_1(s)$ as the top eigenvalue of $D+s A$, so we have
\begin{equation*}
(D+s A)x(s) = \lambda_1(s)x(s)
\end{equation*}
and thus
\begin{equation*}
\ip{y(s)}{(D+s A)x(s)}  = \lambda_1(s)\ip{y(s)}{x(s)}= \lambda_1(s).
\end{equation*}
Differentiating this equation gives
\begin{align*}
  \lambda_1'(s)
  	&= \ip{y'(s)}{(D+s A)x(s)} + \ip{y(s)}{Ax(s)} + \ip{y(s)}{(D+s A)'x(s)}\\
	&= \lambda_1(s)\ip{y'(s)}{x(s)}	+ \ip{y(s)}{Ax(s)} +\lambda_1(s)\ip{y(s)}{x'(s)}\\
	&= \lambda_1(s) \frac{d}{ds} \ip{y(s)}{x(s)} + \ip{y(s)}{Ax(s)}\\
	&= \ip{y(s)}{Ax(s)}.
\end{align*}
In particular, notice that since the entries of $x,y,A$ are all positive, then $\lambda_1(s)$ is increasing in $s$.  Thus $\lambda_1(1)=0$ and $J(1)$ is negative semi-definite.

\end{proof}

\subsection{An Index Theorem -- Instability} \label{index-thm-instability}

We recognize that $\mathcal{S_{\theta}^{\dag}}$ is not necessarily a complete description of $\mathcal{S_{\theta}}$.  In fact, there are likely stationary solutions in $\mathcal{T}$ that are not stable, and thus not in $\mathcal{S_{\theta}}$. To further understanding the stability properties of the stationary solutions, we return to the one-parameter family of matrices $\J(s)=\D+s \A$ and we define the following index.
\begin{define}
We define $n_+(\J)$ to be the number of eigenvalues $\lambda_i(\J)$ in
the open positive half-plane ${\rm Re}(\lambda)>0$ (counted according to algebraic multiplicity).
\end{define}
Our goal is to detect eigenvalue crossings into the right-half plane as $s$ increases from 0 to 1.  A reasonably straightforward linear algebra calculation gives a nice representation of the characteristic polynomial of $\J(s)$.

\begin{lem}\label{lem:det of J(s)}
Define $P_{\J}(s)=\det(\J(s))$. Then we have that
\begin{enumerate}
\item $P_{\J}(s)$ is a quadratic polynomial in $s$ given explicitly by
\begin{equation}
\label{quad_eta_poly}
\begin{split}
     P_{\J}(s)
=  1&+\l( \langle\mb{v}, \mb{D}^{-1}\mb{u}\rangle + \langle\mb{z}, \mb{D}^{-1}\mb{w}\rangle \r)s \\
 & + \l(  \langle\mb{v}, \mb{D}^{-1}\mb{u}\rangle\langle\mb{z}, \mb{D}^{-1}\mb{w}\rangle -  \langle\mb{z}, \mb{D}^{-1}\mb{u}\rangle\langle\mb{v}, \mb{D}^{-1}\mb{w}\rangle \r)s^2.
\end{split}
\end{equation}
\item $s=1$ is a root of  $P_{\J}(s)$, and thus both roots are real.
\item At each root of the polynomial $P_{\J}(s)$ the matrix $\J$ is singular, and $\lambda=0$ is an eigenvalue of $\J(s)$ with algebraic and geometric multiplicity $1$ unless $s=1$ is a double root.   \end{enumerate}
\end{lem}

\begin{proof}
We begin by computing $P_{\J}(s)$.  Consider the eigenvalue problem $(\D+s \A)\mb{x}=\lambda \mb{x}$.  (For notational convenience, we suppress the parametric dependence of $\J$, $\mb{x}$ and $\lambda$ on $s$ throughout.)  Rather than computing directly, we take advantage of the structure of the adjacency matrix $\A$.  So
\begin{align*}
     (\mb{D}+s(\uu\otimes\vv + \w\otimes\z))\mb{x} &= \lambda\mb{x}.
\end{align*}
Note that is is not possible that a right eigenvector of $\D$ is also orthogonal to both $\mb{v}$ and $\mb{z}$.  Recall from Obs.\ \ref{J as rank two perturb} that $\vv= [\cos{(\theta_j-\alpha/2)}]^t$ and $\z= [\sin{(\theta_j-\alpha/2)}]^t$.  Hence for any fixed $\boldsymbol{\theta}$ and $\alpha$, it is impossible that $\bf \langle v,x\rangle$ and $\bf \langle z,x\rangle$ both vanish.  Now choose a right eigenvector $\bf x$ of $\J$ that is not also a right eigenvector of $\D$.  That is, $\mb{ Dx}\neq \lambda \mb{x}$.  Then the eigenvalue problem can be written
\begin{align*}
     (\mb{D}-\lambda)\mb{x} &= -s(\uu\otimes\vv + \w\otimes\z)\mb{x}\\
                    &= -s\mb{u}\langle \mb{v},\mb{x}\rangle - s\mb{w}\langle \mb{z},\mb{x}\rangle
\end{align*}

Since $\bf x$ is not an eigenvector of $\bf D$, we can construct a recursive representation of $\bf x$ of the form
\[
 \mb{x}=-s(\mb{D}-\lambda)^{-1}\mb{u}\langle \mb{v},\mb{x}\rangle-s(\mb{D}-\lambda)^{-1}\mb{w}\langle \mb{z},\mb{x}\rangle
\]
In turn, this equation can be used to find necessary conditions required upon $\lambda$ in terms of constraint equations in the inner products $\bf \langle v,x\rangle$ and $\bf \langle z,x\rangle$.  Taking the left inner product with respect to $\bf v$ yields the equation
\[
\langle \mb{v},\mb{x} \rangle=- s\langle \mb{v},(\mb{D}-\lambda)^{-1}\mb{u}\rangle \langle \mb{v},\mb{x}\rangle - s\langle \mb{v},(\mb{D}-\lambda)^{-1}\mb{w}\rangle\langle \mb{z},\mb{x}\rangle.
\]
Considering this equation as in the unknown scalars $\bf \langle v,x\rangle$ and $\bf \langle z,x\rangle$, it can be written
\[
(1+s\langle \mb{v},(\mb{D}-\lambda)^{-1}\mb{u}\rangle)\langle \mb{v},\mb{x}\rangle+s\langle \mb{v},(\mb{D}-\lambda)^{-1}\mb{w}\rangle\langle \mb{z},\mb{x}\rangle  = 0.
\]
Similarly, the left inner product with respect to $\bf z$ yields
\[
s\langle \mb{z},(\mb{D}-\lambda)^{-1}\mb{u}\rangle\langle \mb{v},\mb{x}\rangle +(1+s\langle \mb{z},(\mb{D}-\lambda)^{-1}\mb{w}\rangle)\langle \mb{z},\mb{x}\rangle  = 0
\]
and together we have a system of two equations in two unknowns. Again, for a eigenvector $\mathbf{x}$, $\bf \langle v,x\rangle$ and $\bf \langle z,x\rangle$ can not both be zero. Thus, the coefficient matrix
\[
 \left[\begin{array}{cc}
 1+s\langle \mb{v},(\mb{D}-\lambda)^{-1}\mb{u}\rangle  & s\langle \mb{v},(\mb{D}-\lambda)^{-1}\mb{w}\rangle \\
 s\langle \mb{z},(\mb{D}-\lambda)^{-1}\mb{u}\rangle &  1+s\langle \mb{z},(\mb{D}-\lambda)^{-1}\mb{w}\rangle
\end{array}\right]
\]
must be singular.  In other words, if $\lambda$ is an eigenvalue of $\bf J$, then $\lambda$ must satisfy the condition
\[
\det{\left(\left[\begin{array}{cc}
1+s\langle \mb{v},(\mb{D}-\lambda)^{-1}\mb{u}\rangle  & s\langle \mb{v},(\mb{D}-\lambda)^{-1}\mb{w}\rangle \\
s\langle \mb{z},(\mb{D}-\lambda)^{-1}\mb{u}\rangle & 1+s\langle \mb{z},(\mb{D}-\lambda)^{-1}\mb{w}\rangle
\end{array}\right]\right)} = 0.
\]
We define $P_{\J}(s)$ to be this determinant.

When $s=1$, we are considering the original Jacobian matrix $\eqref{J}$.  Note that by the definition of $\J$, $\J$ always has a zero eigenvalue, with right eigenvector $(1,1,\ldots,1)$.  Thus $s=1$ is a root of $P_{\J}(s)$ and both roots of $P_{\J}(s)$ must be real.
\end{proof}

While the non-self-adjoint nature of the operator makes it difficult to get results as sharp as those in the standard Kurammoto model, we can establish a
sufficient condition for instability. The main observation here is that, since $\J(s)$ is real the eigenvalues
occur in complex conjugate pairs, and thus the index $n_+(\J(s))$ can
change in the following ways
\begin{itemize}
\item  $n_+(\J(s))$ can change by one when a real eigenvalue passes through the origin.
\item  $n_+(\J(s))$ can change by two when a complex conjugate pair of eigenvalues crosses the imaginary axis.
\end{itemize}
The first of these possibilities is easy to detect, as it is signalled by the
vanishing of $P_{\J}(s)$. The second is not so easy to detect, thus motivating us to count modulo two.

\begin{thm}
Assume $D$ is non-singular and let $n_+(D)$ be the number of positive
eigenvalues of the diagonal
matrix $D$ where the diagonal entries are of the form $D_{ii}=-\sum_{j}\cos{(\theta_j-\theta_i-\alpha)}$. Let $n_R$ be the number
of roots of the quadratic $P_{\J}(s)$, equation \eqref{quad_eta_poly}, in the open interval $(0,1)$.
Let $n_+(D+A)$ be the number of
eigenvalues of the linearized operator in the open right half-plane ${\rm Re}(\lambda)>0$. Finally assume that $s=1$ is a simple root of  $\det(D+s A)$,
and that $\lambda_0(s)$ is the eigenvalue branch with $\lambda_0(1)=0$. Then
we have the equality
\[
(-1)^{n_+(D+A)-n_+(D)-n_R} = \sign\left(\l.\frac{d\lambda_0}{ds}\r|_{s=1}\right)
\]
In particular we have the following sufficient conditions for instability
\begin{enumerate}[(i)]
    \item $n_R=0$, $n_+(D)$ even and $\frac{d\lambda_0}{ds}<0$;
    \item $n_R=0$, $n_+(D)$ odd and $\frac{d\lambda_0}{ds}>0$;
    \item $n_R=1$, $n_+(D)$ even and $\frac{d\lambda_0}{ds}>0$;
    \item $n_R=1$, $n_+(D)$ odd and $\frac{d\lambda_0}{ds}<0$.
\end{enumerate}
\end{thm}

\begin{remark}
 The mod two nature of the count arises from the fact that we can have
 complex conjugate pairs of eigenvalues crossing from the left
 half-plane to the right half-plane.  In the classical Kuramoto case we always have that $\frac{d\lambda_0}{ds}>0$ and the
count modulo two becomes an actual count
\[
n_+(D+A)=n_R+ n_+(D).
\]
In particular the stability region for classical Kuramoto is defined by the curve where
$n_+(D)=0$ and $n_R$ transitions from $0$ to $1$ -- essentially the
boundary of the set defined by condition $(iii)$ above. We will see
later in the numerics section that it appears numerically that
boundary of the stable
region is always defined by transition to one of the conditions listed above.
\end{remark}

\begin{proof}
The proof here is simply a collection of prior results and comments.  We return to $\mathbf{J}(s)$, equation~\eqref{J of s operator}, the continuation in $s$ from the diagonal
 matrix $\mathbf{D}$ to the true case of interest, $\mathbf{D+A}$.  The basic observation is that the number of eigenvalues in the left half-plane changes by one when a real eigenvalue passes through the origin and changes by two when a complex conjugate pair of eigenvalues crosses through the axis. Since we can detect real crossings we can easily get a count modulo two.

In Lemma~\ref{lem:det of J(s)}, we derived equation~\eqref{quad_eta_poly} which showed that $P_{\J}(s)=\det(D + s A)$ is a real quadratic function of $s$
and that both roots of $\det(D +  A)=0$ are real. We have then that at a root $s_0$ of  $\det(D + s A)$ the null-space is simple and the crossing transverse,
$\frac{d\lambda_0}{ds}\neq 0$, unless $s=1$ is a double root.  Thus, whenever
\[
(-1)^{n_+(D+A)} = (-1)^{-n_+(D)+n_R}\sign\left(\l.\frac{d\lambda_0}{ds}\r|_{s=1}\right)<0
\]
the system must be unstable.
\end{proof}

\section{Numerical Results} \label{sec:numerics}

\subsection{Visualization of the small $N$ oscillators model}

We begin with a visualization of the three oscillator model, $N=3$.  For any $\alpha$, the map $\mathbf{f}$ is rotationally invariant in the configuration space $\boldsymbol{\theta}$.  Thus the image of the reduced configuration space $\mathcal{T}:=\T^3\bigcap\mathbf{1}^{\perp}$ under the map $\mathbf{f}$ will be a 2-dimensional surface in the 3-dimensional $\boldsymbol{\omega}$-space.  Figure~\ref{fig:N3_ConfigurationSpace} is the surface $\mathbf{f}(\mathcal{T})$ associated with the fixed detuning parameter $\alpha=\pi/6$.
\begin{figure}[th]
\begin{center}
\includegraphics[width=0.4\textwidth]{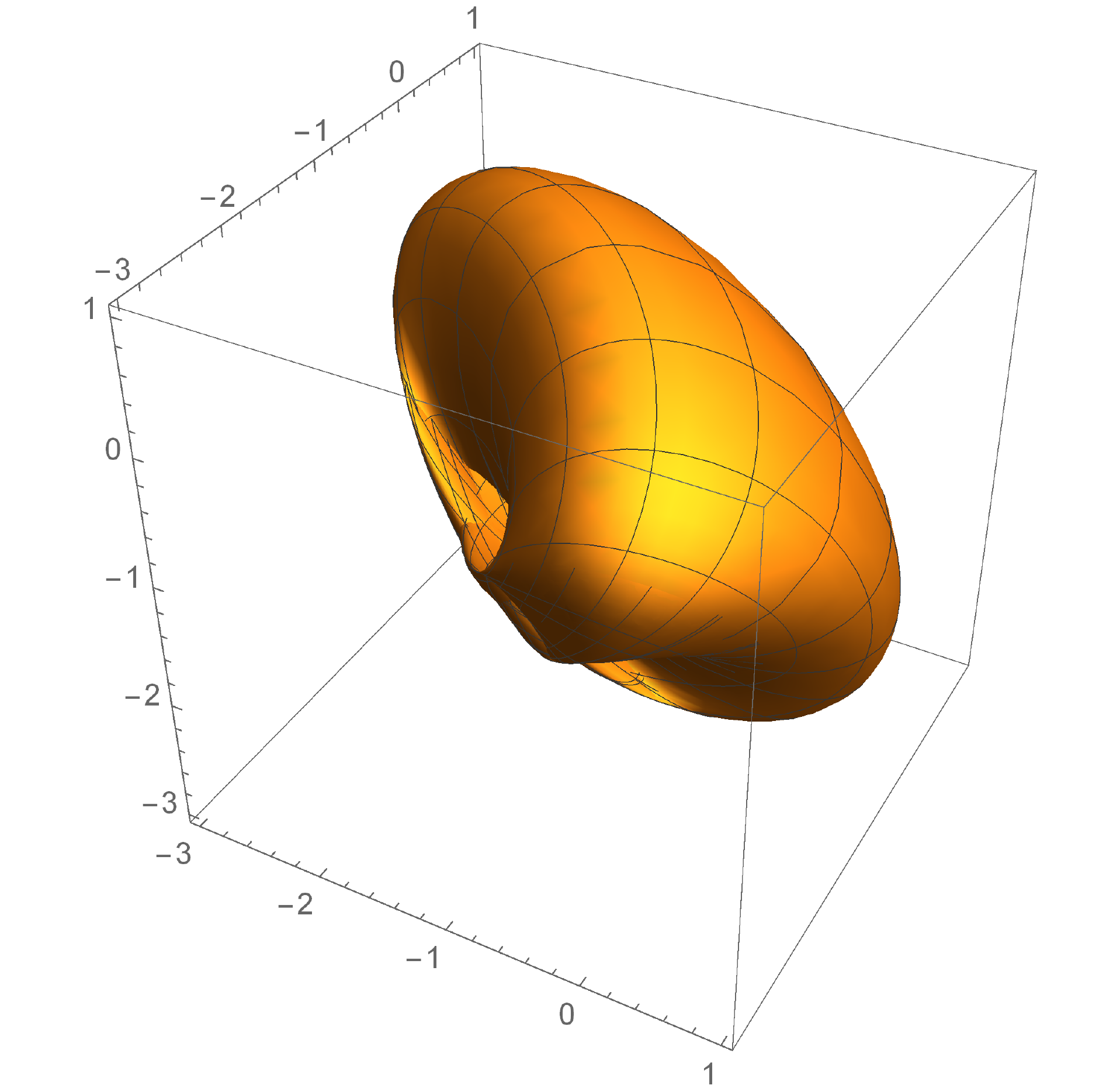}\includegraphics[width=0.4\textwidth]{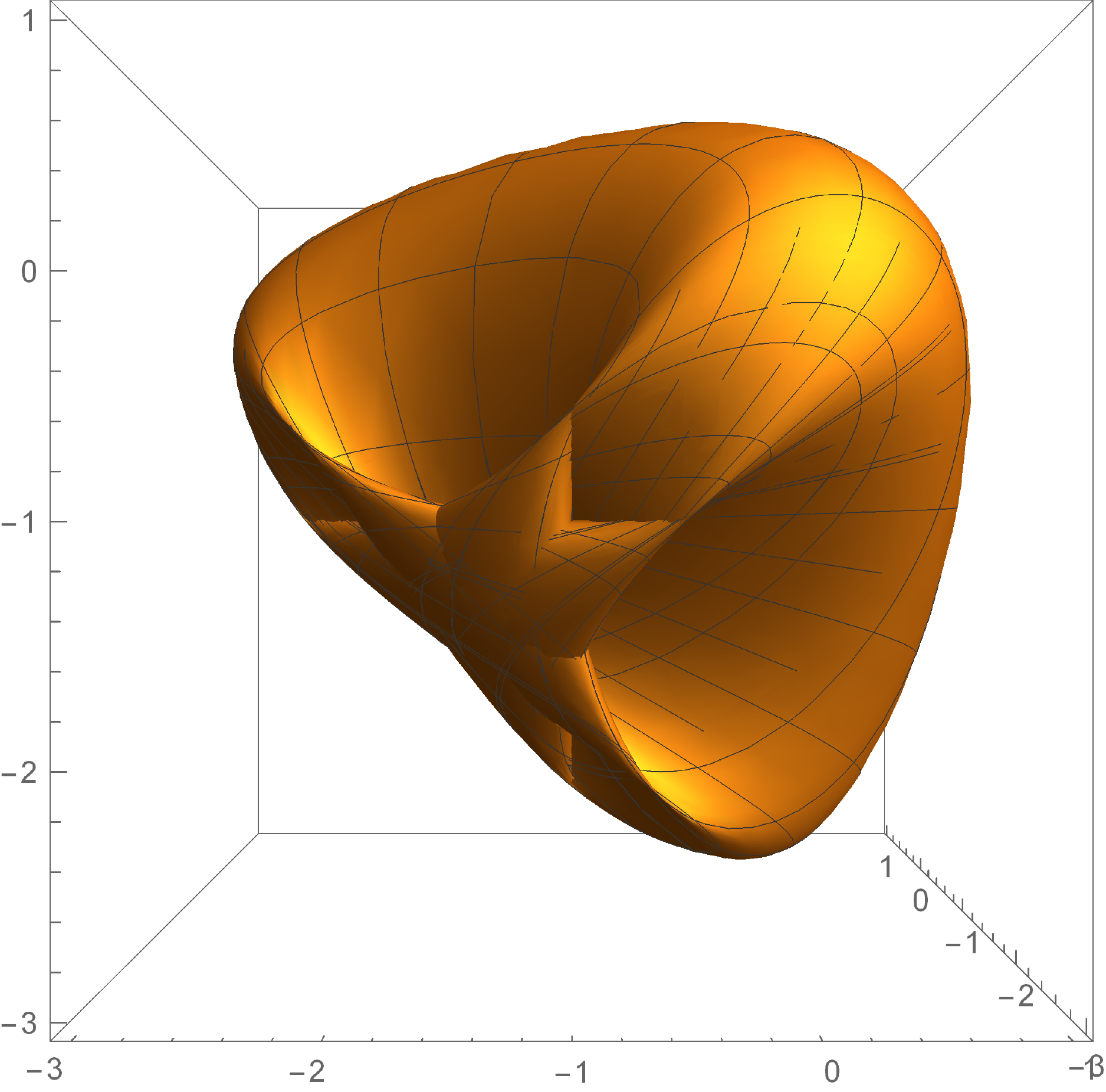}
\end{center}
\caption{$N=3$ with $\alpha=\pi/6$, two views of the image of the $\mathcal{T}$ under $\mathbf{f}$}
\label{fig:N3_ConfigurationSpace}
\end{figure}
We are interested in the portion of this surface that corresponds to $\mathcal{S}_{\omega}:=\mathbf{f}(\mathcal{S}_{\theta})$, the $\bomega$ that give rise to a phase-locked solutions.  To see this, we return to the pre-image $\mathcal{T}$.  Again the $\mathbf{1}^{\perp}$ plane is spanned by $\mathbf{e}_1=(1,-1,0)/\sqrt{2}$ and $\mathbf{e}_2=(1,1,-2)/\sqrt{6}$.  In the local coordinates defined by $\mathbf{e}_1$ and $\mathbf{e}_2$, the phase diagram for three oscillators (where $\alpha=\pi/6$) is summarized in Figure~\ref{fig:N3_ConfSpace_StabilityColoring}.
\begin{figure}[th]
\begin{center}
\includegraphics[width=0.4\textwidth]{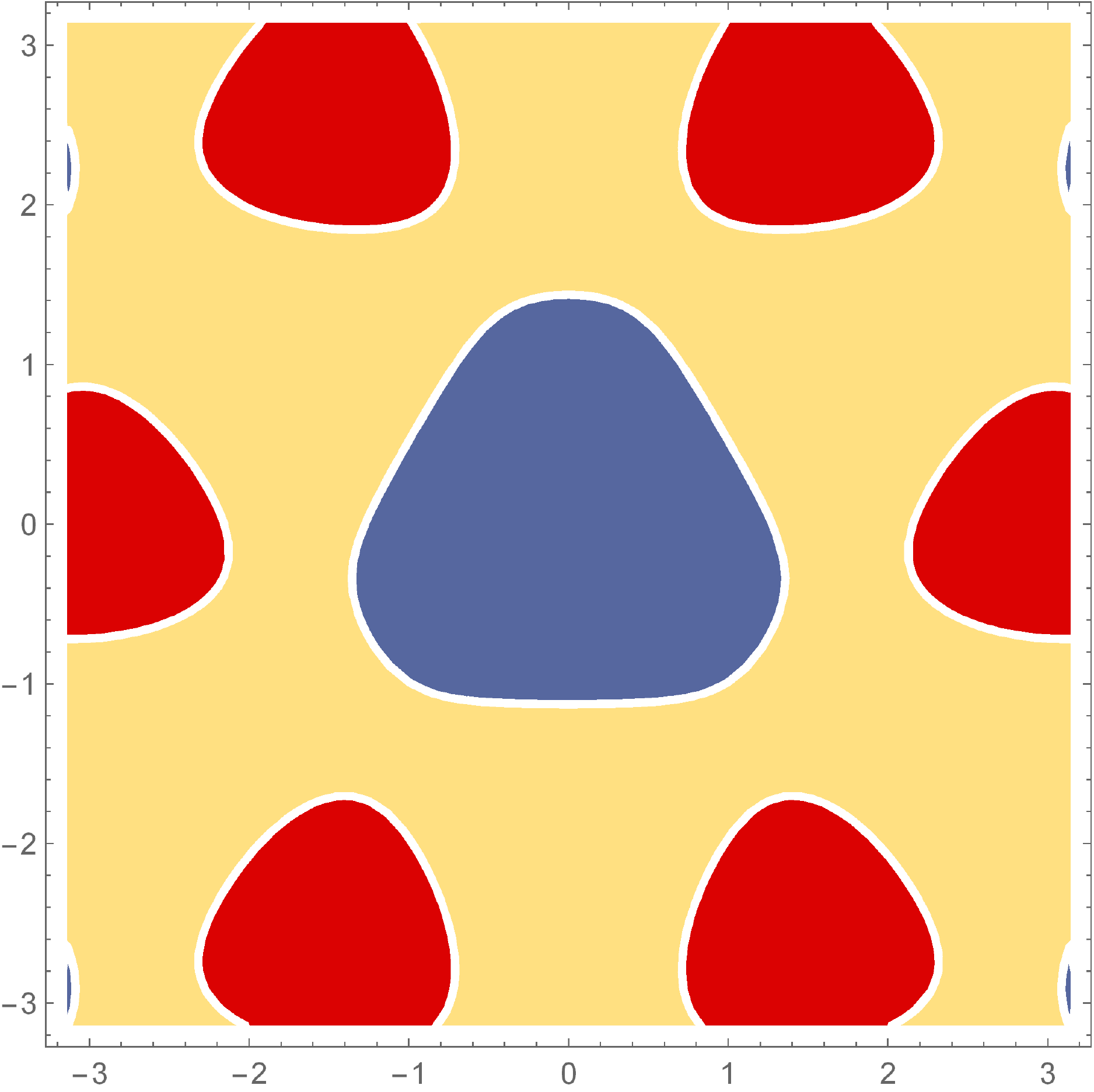}\hspace{0.025\textwidth}\includegraphics[width=0.4\textwidth]{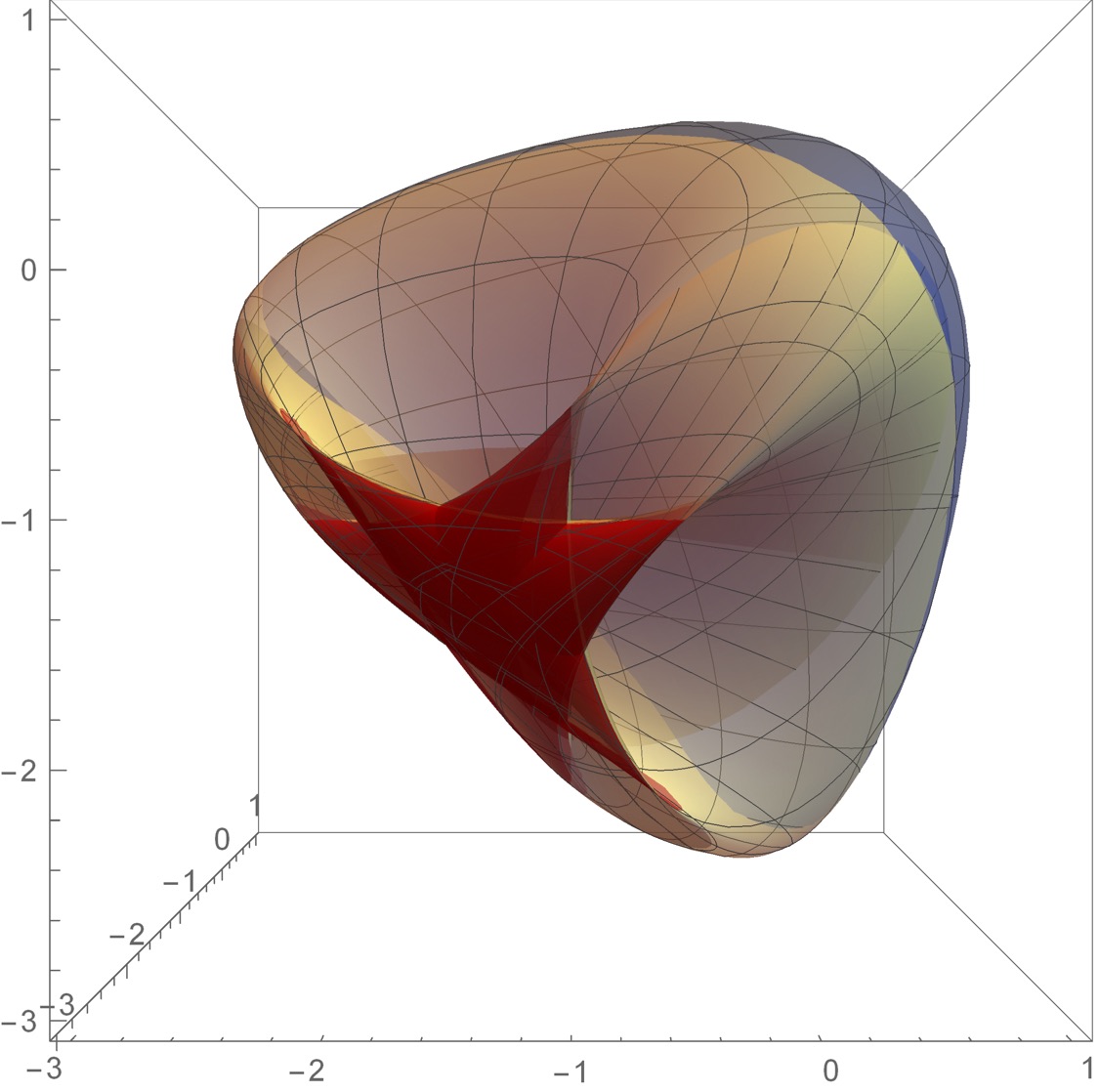}
\end{center}
\caption{$N=3$ with $\alpha=\pi/6$, $\mathbf{1}^{\perp}$ and $\mathbf{f}(\mathbf{1}^{\perp})$ colored by stability (blue is stable)}
\label{fig:N3_ConfSpace_StabilityColoring}
\end{figure}
In this image, $\mathcal{S}_{\theta}$ and $\mathcal{S}_{\omega}$ are paired side-by-side.  The blue regions corresponds to the stable region in local coordinates, where the Jacobian of $\mathbf{f}$ is negative semi-definite with a one dimensional kernel.  The gold and red regions correspond to when the Jacobian of $\mathbf{f}$ has one or two unstable eigen-direction, respectively.  For added clarity, we have the frequency space de-constructed by index in Figure~\ref{fig:N3_ConfSpace_ColorDeconstruction}.  We note that stable set here corresponds exactly to the bottom of the $\alpha=\pi/6$ surface in the orientation of Figure~\ref{fig:ShadowGraphs} (when we discussed the frequency region $\roa$ in Section~\ref{subsec:Stabilty is different in KS}).
\begin{figure}[th]
\begin{center}
\includegraphics[width=0.3\textwidth]{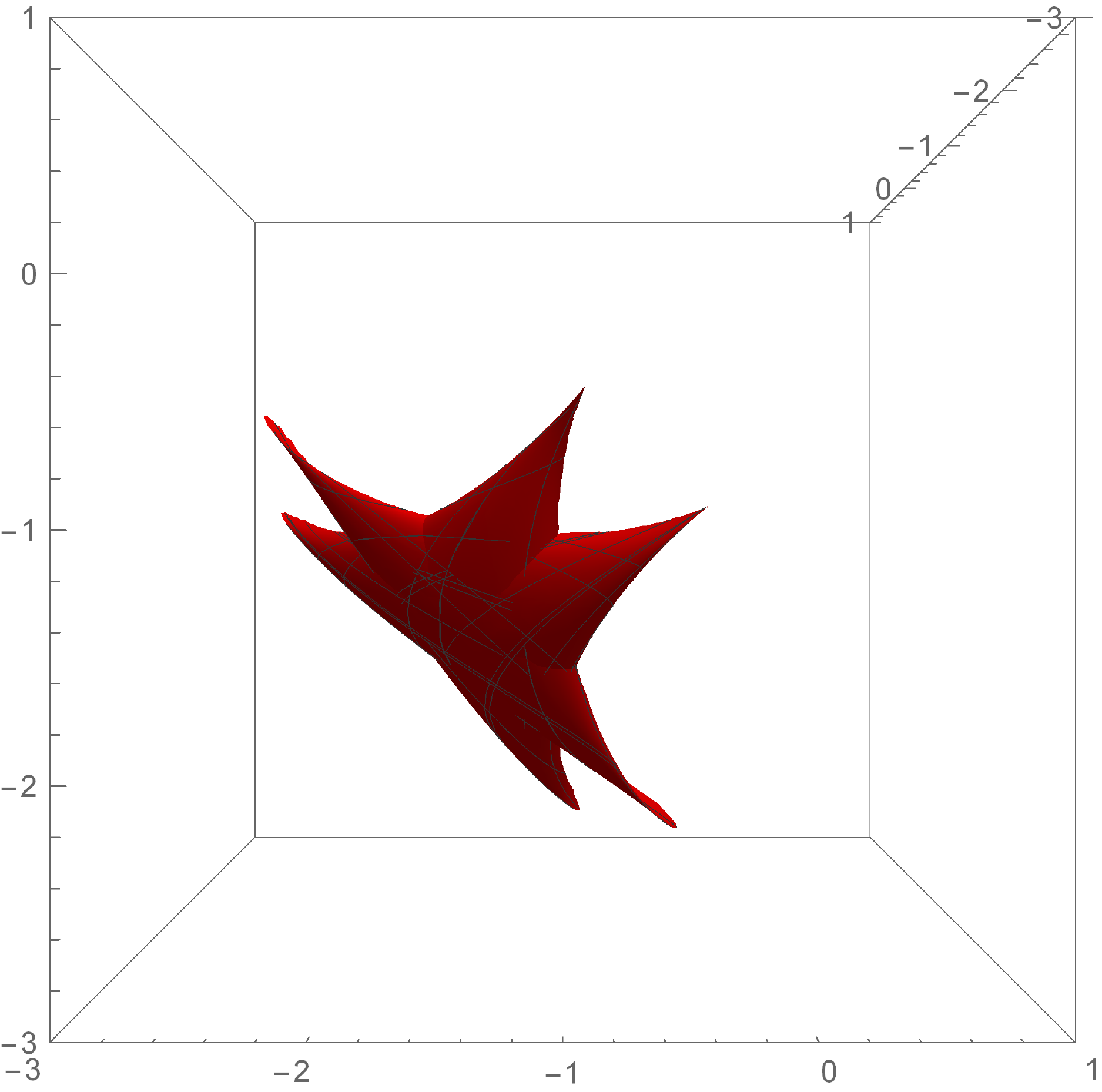}\includegraphics[width=0.3\textwidth]{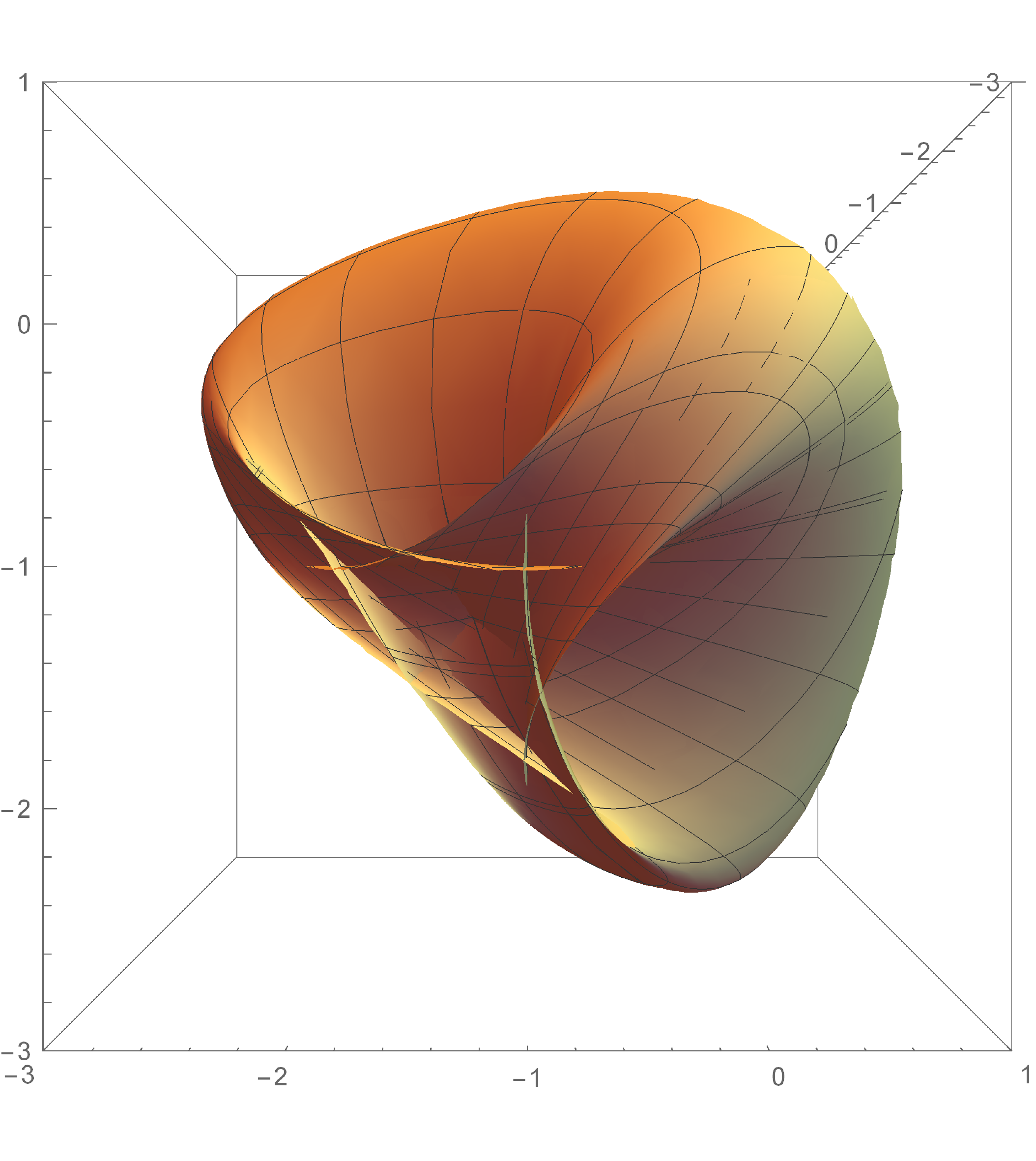}\includegraphics[width=0.3\textwidth]{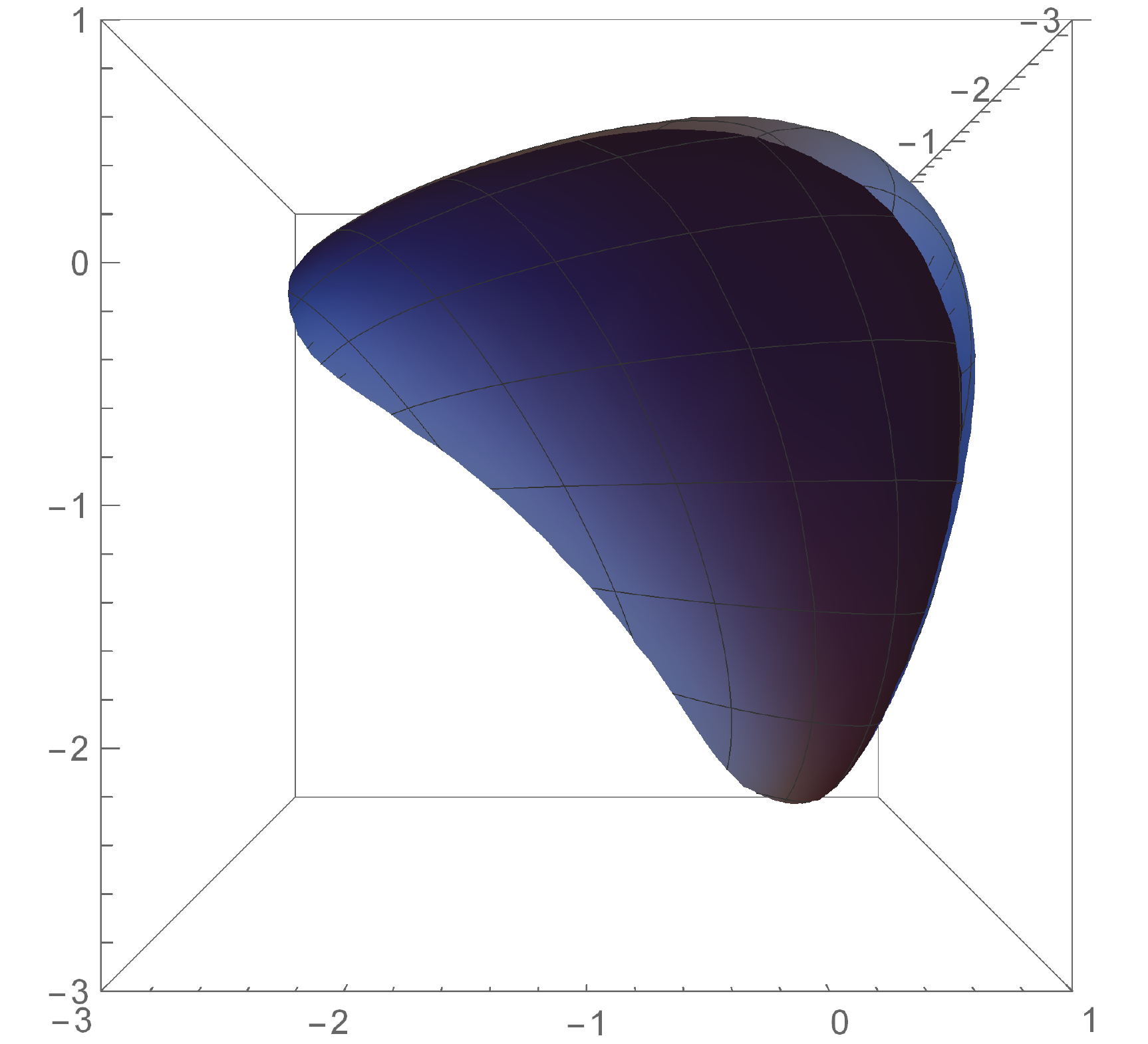}
\end{center}
\caption{$N=3$ with $\alpha=\pi/6$, $\mathbf{f}(\mathbf{1}^{\perp})$ deconstructed by index}
\label{fig:N3_ConfSpace_ColorDeconstruction}
\end{figure}

 Determining the probability for a Kuramoto system to admit a phased locked solution when $\bomega$ is randomly assigned is directly related to understanding the size of the frequency space $\mathcal{S}_{\omega}$.  For standard Kuramoto, $\mathcal{S}_{\theta}$ is a convex set that is invariant under the actions of the dihedral group of order $N$ (see \cite{Bro.DeV.Park.2012}).  Moreover, these geometric properties are preserved under the map $\mathbf{f}$ when $\alpha$ is zero.  This is due to the fact that when $\alpha=0$, $\mathbf{f}$ is an odd function and the reflections $\bomega\rightarrow -\bomega$ and $\btheta\rightarrow-\btheta$ admit another stable solution.  In the end, excellent estimates exist for the size of $\mathcal{S}_{\omega}$ for the standard model. This geometry does not hold for the Kuramoto-Sakaguchi model.  In Figure~\ref{fig:N3_Boundary of Stable Region}, we have the boundary of $\mathcal{S}_{\theta}$ for values of $\alpha$ ranging from zero to $5\pi/12$.  The ``hexagonal" red curve corresponds to standard Kuramoto.  All other boundary curves loose the extra geometric structure and possess only triangular symmetry.  In fact, it is easy to see the same behavior in the $N=4$ model.  In Figure~\ref{fig:N4_FreqSpaces}, the corresponding pictures are all in local coordinates as we are unable to embed $\mathcal{S}_{\omega}$ in its natural space.  None the less, the symmetry and reduction of dihedral order is easily prevalent.  In keeping with earlier conventions, the blue regions in Figure~\ref{fig:N4_ConfigSpaces} correspond to the stable region $\mathcal{S}_{\theta}$ in local coordinates for varying $\alpha$.   We know that standard Kuramoto possesses octahedral symmetry, while we can see that all other configuration spaces correspond to tetrahedral symmetry.  Moreover, in Figure~\ref{fig:N4_FreqSpaces} this symmetry is preserved (in local coordinates) under the mapping $\mathbf{f}$.

As $\mathcal{S}_{\omega}$ no longer resides in the $\mathbf{1}^{\perp}$ plane, we can't use convexity to estimate its size.  But at the outset of our study, there was no reason to believe $\mathcal{S}_{\theta}$ would not be so and we hoped to use the size of $\mathcal{S}_{\theta}$ to estimate the size of the stable region.  For large $\alpha$, the convexity of $\mathcal{S}_{\theta}$ is lost.  In Figure~\ref{fig:N3_Boundary of Stable Region}, the fuchsia colored boundary curve corresponds to $\alpha= 7\pi/24$ and the associated $\mathcal{S}_{\omega}$ is clearly no longer convex.  The same loss of convexity can be seen in the four oscillator model.  In Figure~\ref{fig:N4_ConfigSpaces},the configuration space corresponding to $\alpha=2\pi/7$ is not convex.  In the end we were unable to develop a geometric approach to estimating the size of the stable region.  In Section~\ref{subsec:highernumerics}, we appeal to a purely numeric approach to estimate the size.

\begin{figure}[th]
\begin{center}
\includegraphics[width=0.3\textwidth]{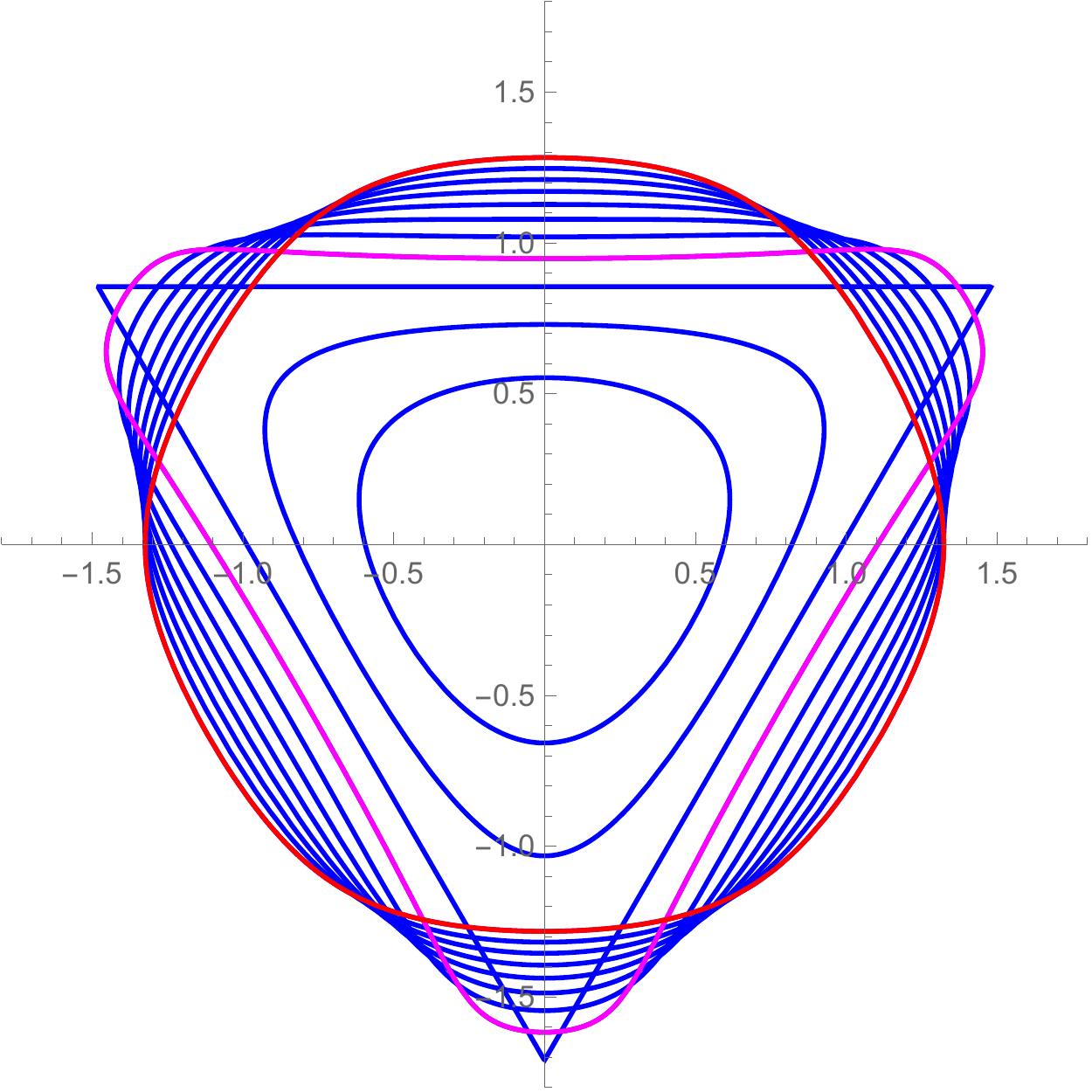}\includegraphics[width=0.3\textwidth]{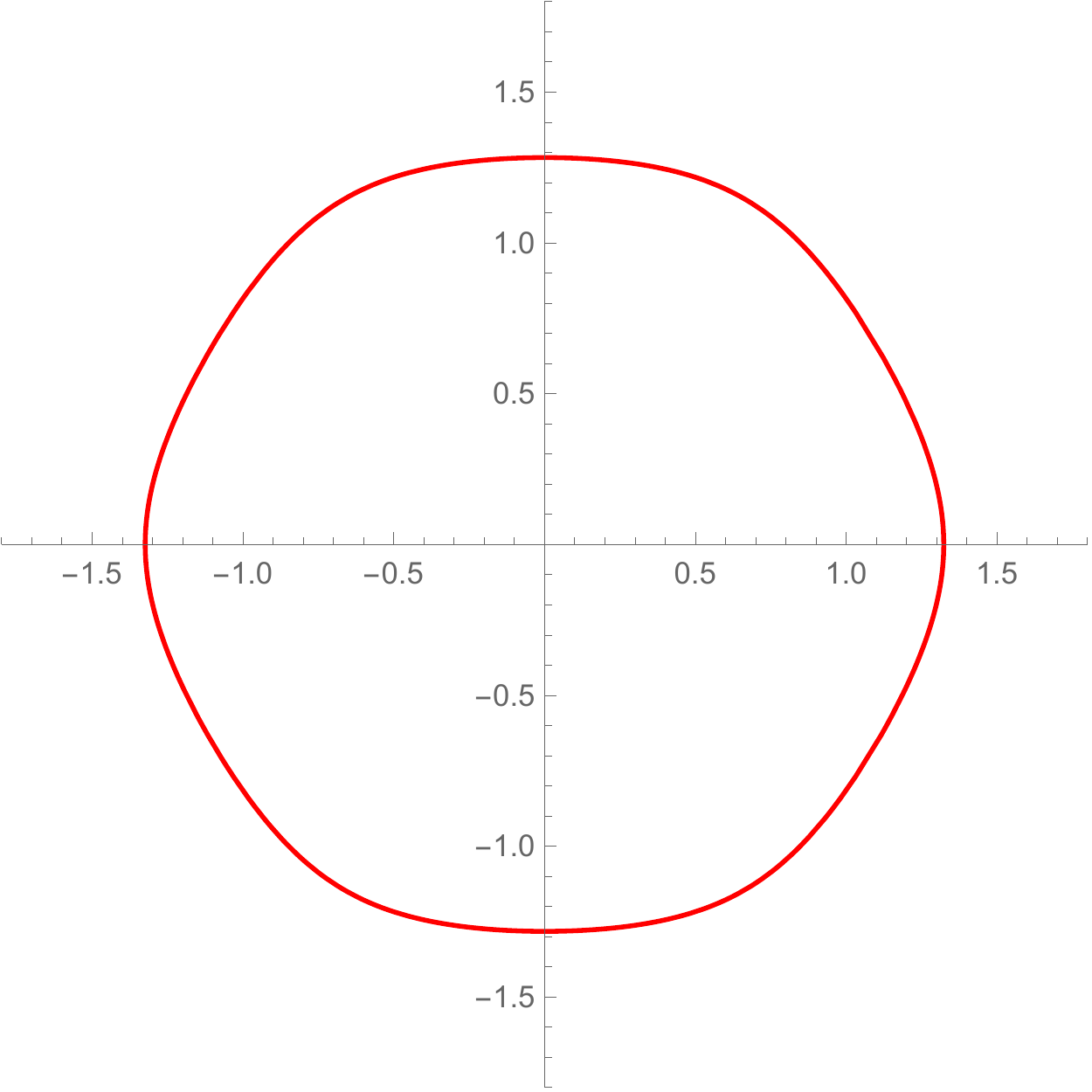}\includegraphics[width=0.3\textwidth]{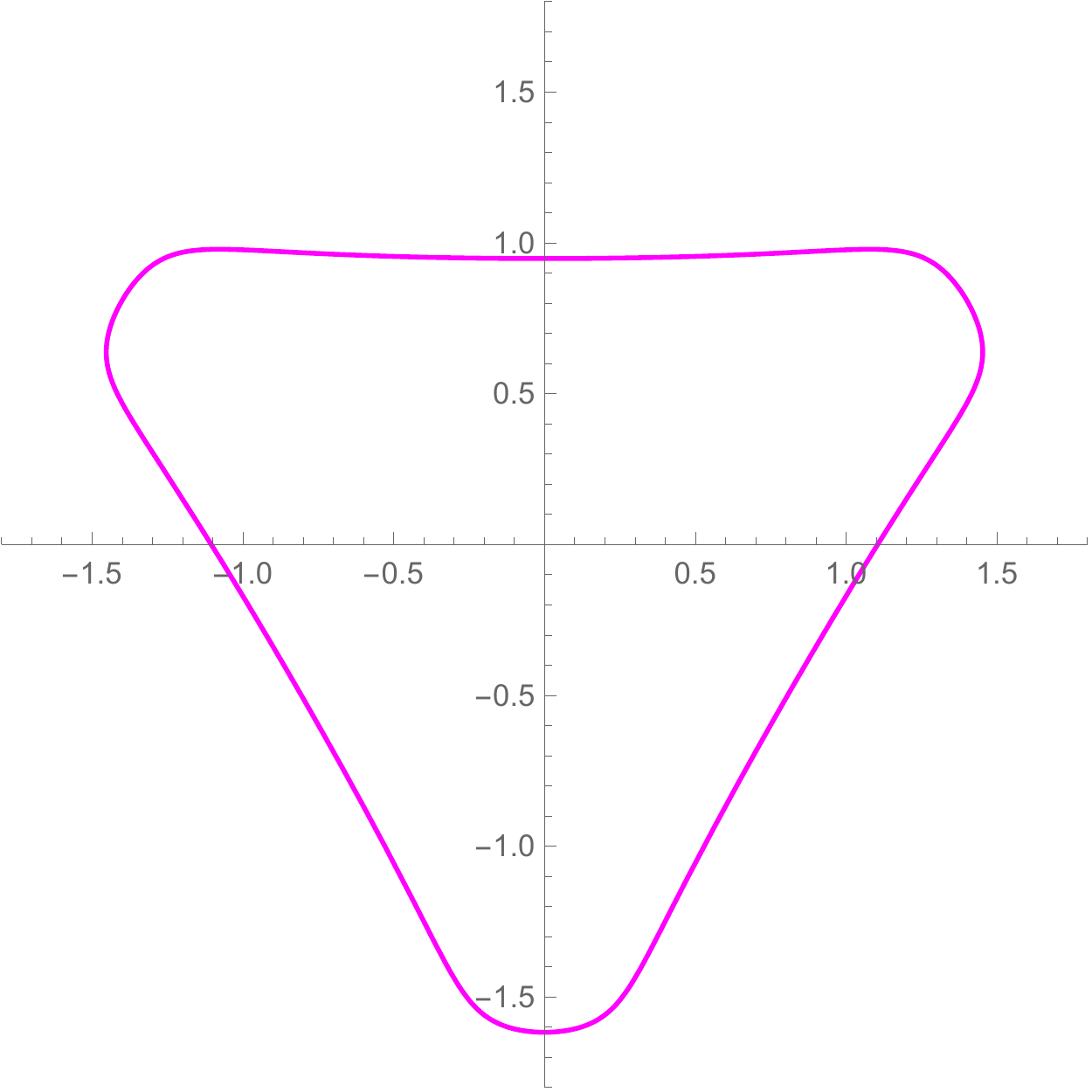}
\end{center}
\caption{$N=3$, the boundary of the stable region in $\mathbf{1}^{\perp}$ for varying values of $\alpha$ }
\label{fig:N3_Boundary of Stable Region}
\end{figure}

\begin{figure}[th]
\begin{center}
\includegraphics[width=0.95\textwidth]{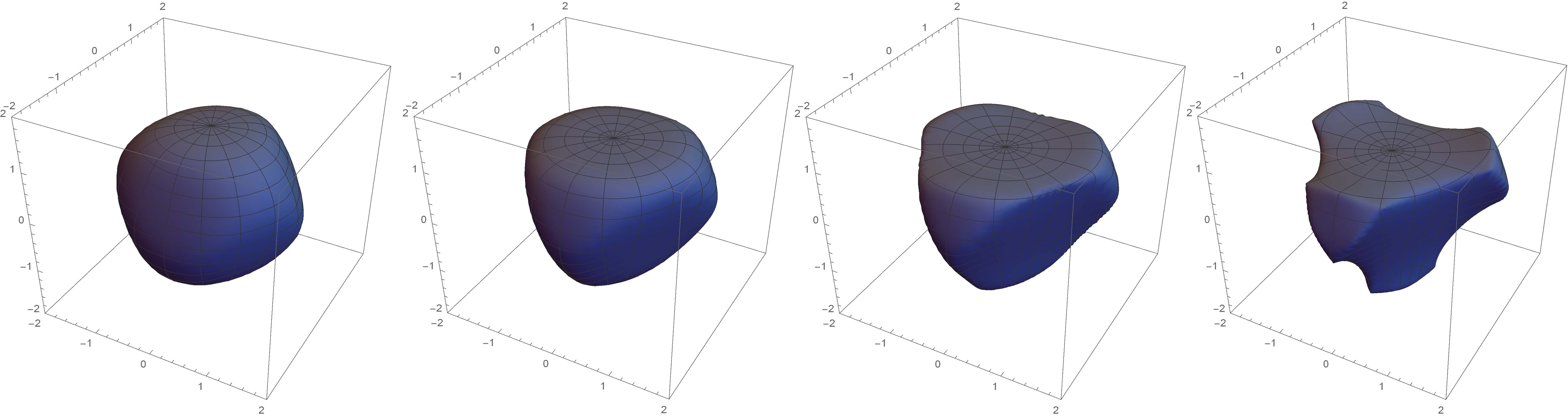}

\end{center}
\caption{$N=4$, The region of stability configurations $\mathcal{S}_{\theta}$ for $\alpha\in\{0, \pi/6, \pi/4, 2\pi/7\}$}
\label{fig:N4_ConfigSpaces}
\end{figure}

\begin{figure}[th]
\begin{center}
\includegraphics[width=0.95\textwidth]{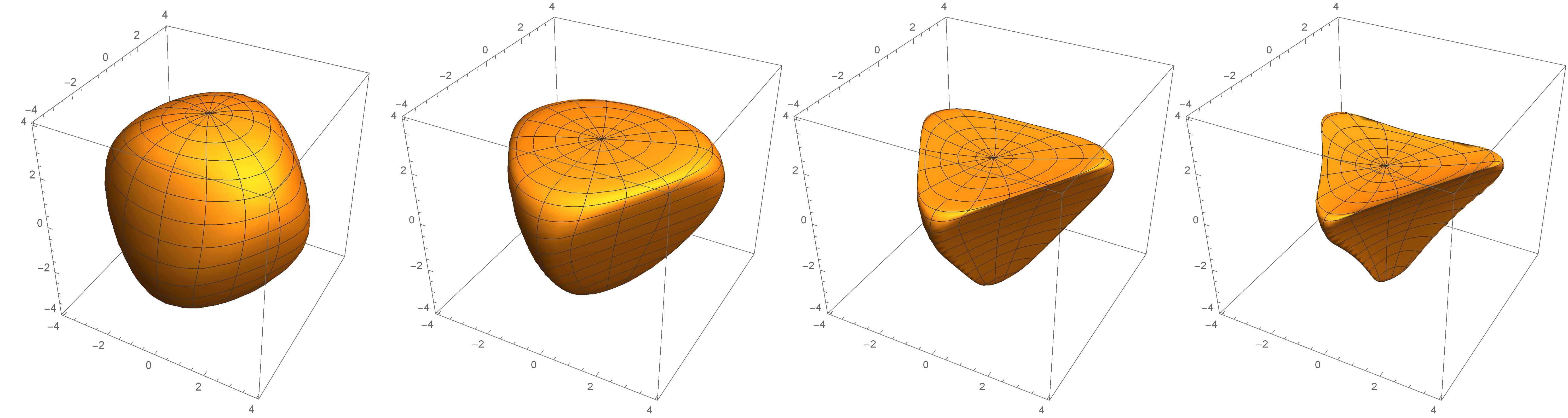}
\end{center}
\caption{$N=4$, The region of stability frequencies $\mathcal{S}_{\omega}=\mathbf{f}(\mathcal{S}_{\theta})$ for $\alpha\in\{0, \pi/6, \pi/4, 2\pi/7\}$}
\label{fig:N4_FreqSpaces}
\end{figure}
stuff

\subsection{Visualization of the Instability Index for Three Oscillators}

For the three oscillator model, we give a sequence of numerical plots
depicting the count, modulo two, of the dimension of the unstable
manifold. This is depicted in Figure (\ref{fig:mod2stuffs}).
The graphs are given in (mean zero) configuration space: $\theta =
x\left(\frac1{\sqrt{2}},-\frac{1}{\sqrt{2}},0\right) + y
\left(\frac{1}{\sqrt{6}},\frac{1}{\sqrt{6}},-\frac{2}{\sqrt{6}}\right)$. The
configuration plane is colored white if the dimension of the unstable
manifold is even (including zero, the stable case) and is shaded if
the configuration has an odd dimensional unstable manifold (obviously
always unstable). The four subgraphs represent different $\alpha$
values: $\alpha = 0,\frac{\pi}{6},\frac{\pi}{3},\frac{2\pi}{5}$.
What is interesting is that the count modulo two of
the number of unstable eigenvalues appears to always capture the most
important transition, that from stability to instability. In each of
pictures the central white region is stable, and is surrounded by six
regions where there are two eigenvalues of positive real part. For
most values of $\alpha$ these regions do not touch showing that as one
varies the configuration the transition from stability to instability
occurs by a single real eigenvalue crossing from the left to the right
half-lines, a transition that is always detected by our theorem. For a
single value of $\alpha=\frac{\pi}{3}$ the stable region touches the
regions with two unstable eigenvalues on a co-dimension two set
(three isolated points), but the boundary of the stable region is
still defined by the curve representing a single real eigenvalue
crossing. For all other values of $\alpha$ the stable
region appears to be the region containing the origin and bounded by
the curves where the instability index changes from even to odd.
There is no obvious region why a configuration could not go unstable
by having a complex conjugate pair of eigenvalues cross from the left
half-plane to the right, but we have not observed this occurring.

\begin{figure}[th]
\begin{center}
\includegraphics[width=0.9\textwidth]{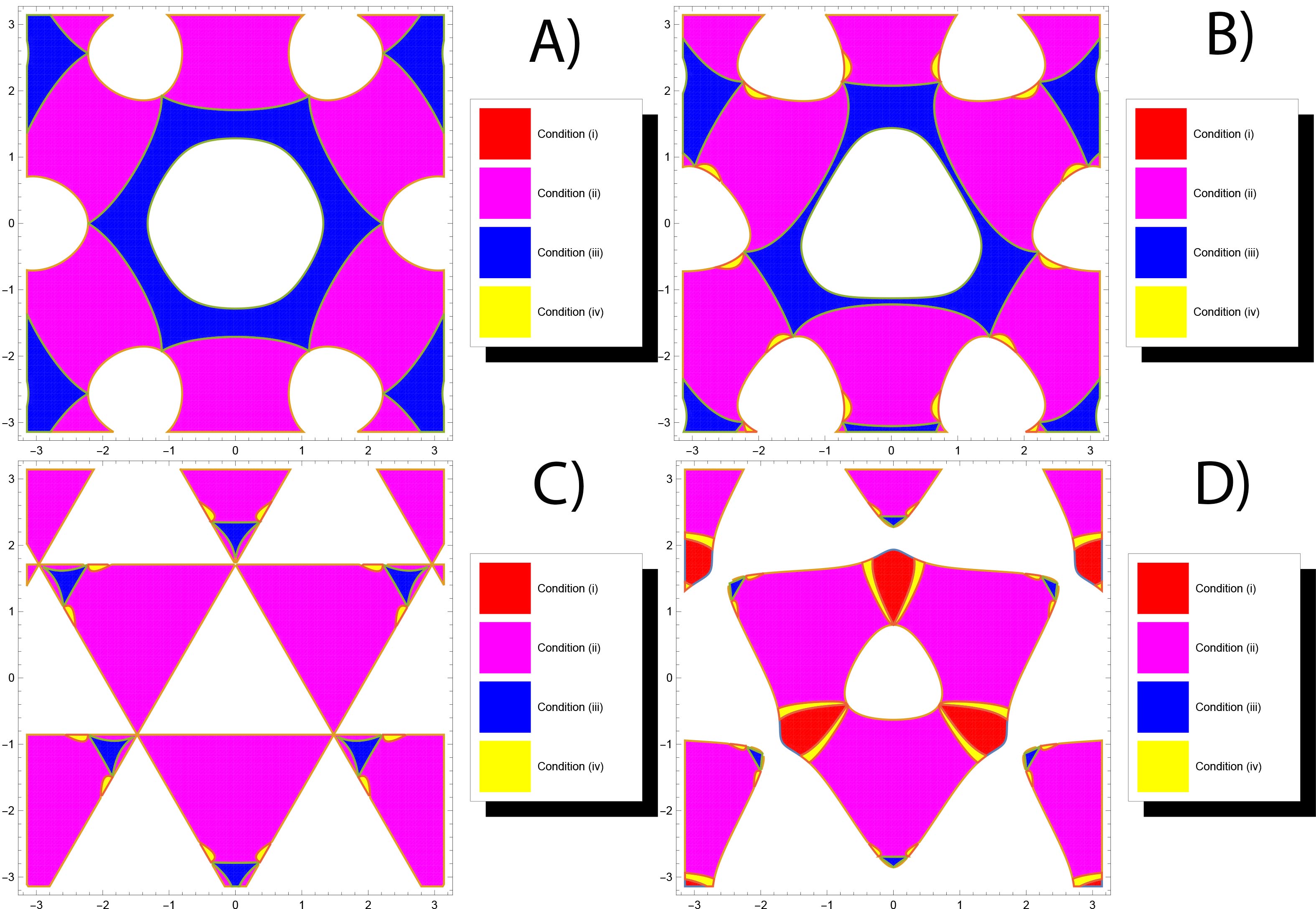}
\end{center}
\caption{This figure (color online) depicts the mod 2 instability
  count in configuration space for various values of $\alpha$ ($\alpha \in
  \{0,\frac{\pi}{6},\frac{\pi}{3},\frac{2\pi}{5}\}$ for subfigures (A)---(D) resp.) The white region
  indicates regions where the corresponding phase-locked solutions
  have an even dimensional unstable manifold, while the shaded regions
indicate an odd dimensional unstable manifold. }

\label{fig:mod2stuffs}
\end{figure}

\subsection{Higher-dimensional numerics}\label{subsec:highernumerics}

It is difficult to visualize the shape of the stable region when $N$ is large, but we can compute its volume.  In this section, we present a few figures showing how the volume of the stable region varies with respect to $N$ and $\alpha$.

The basic method used in this section is of Monte Carlo type, but a direct Monte Carlo simulation will not be useful here.  When $N$ is large, we expect the stable region to scale exponentially with respect to some fixed volume; as an example, imagine that we can bound the stable region in some ball in some $\ell^p$ norm.  Unless the stable region is just lucky enough to fill out most of this ball for large $N$ (and this will only occur if the region is well-represented in the ``corners'' of the ball), then the vast majority of our samples will be outside of the stable region.

To fix this issue, we do a stratified sampling approach.  More specifically, let us say that we're given the parameter $\texttt{numStrata}$ and $\texttt{numSamples}$.  We then define $t_k = (k/$\texttt{numStrata}$)\pi$, and define the region $R_k = [-t_k,t_k]^N$.  We then choose $\texttt{numSamples}/\texttt{numStrata}$ samples uniformly in $R_k\setminus R_{k-1}$, count the number inside the stable region (we can compute $n_+(\J)$ for each sample), and then weight these samples by $\mathrm{vol}(R_k\setminus R_{k-1})$.

\begin{figure}[th]
\begin{center}
\includegraphics[width=0.4\textwidth]{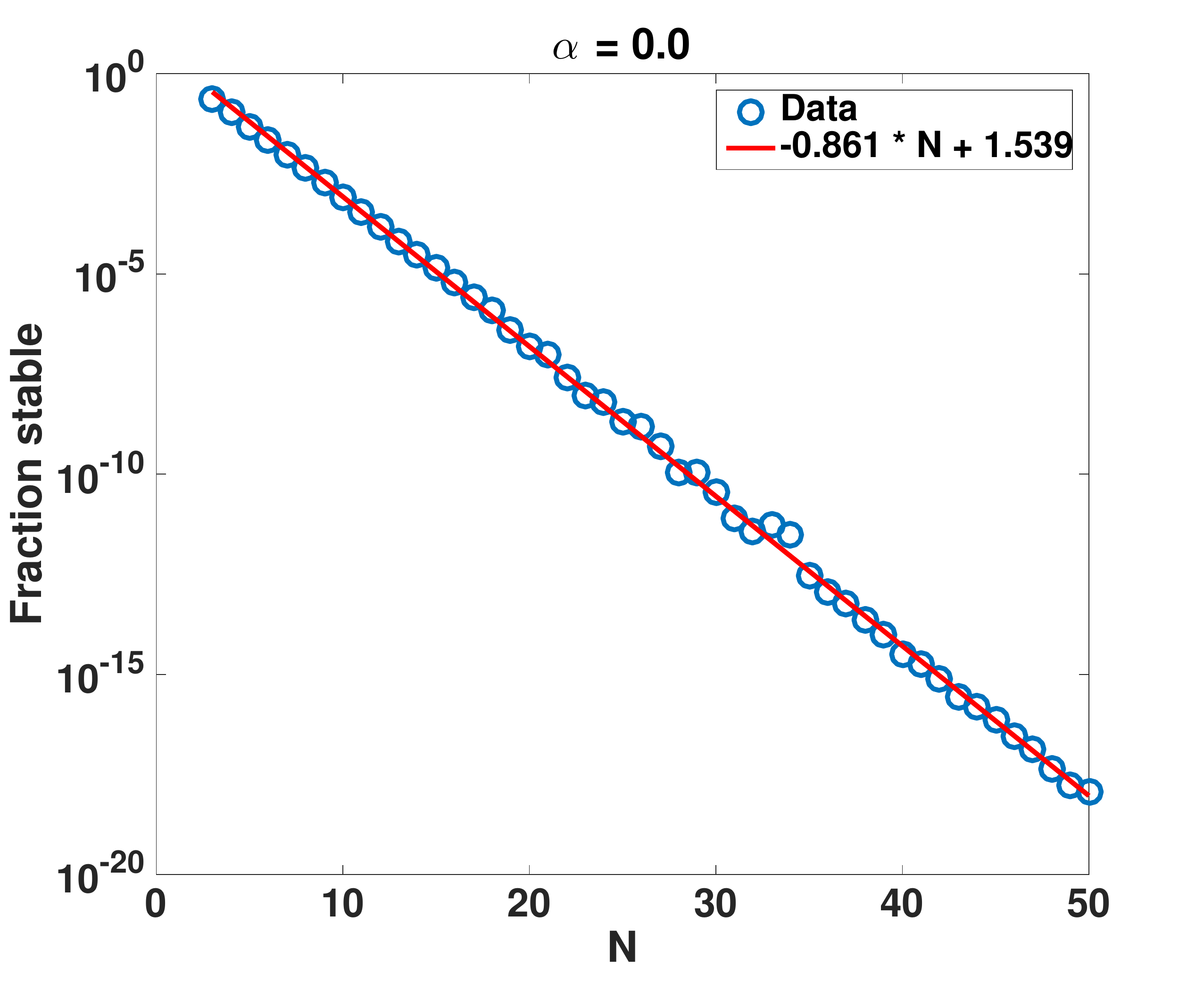}\includegraphics[width=0.4\textwidth]{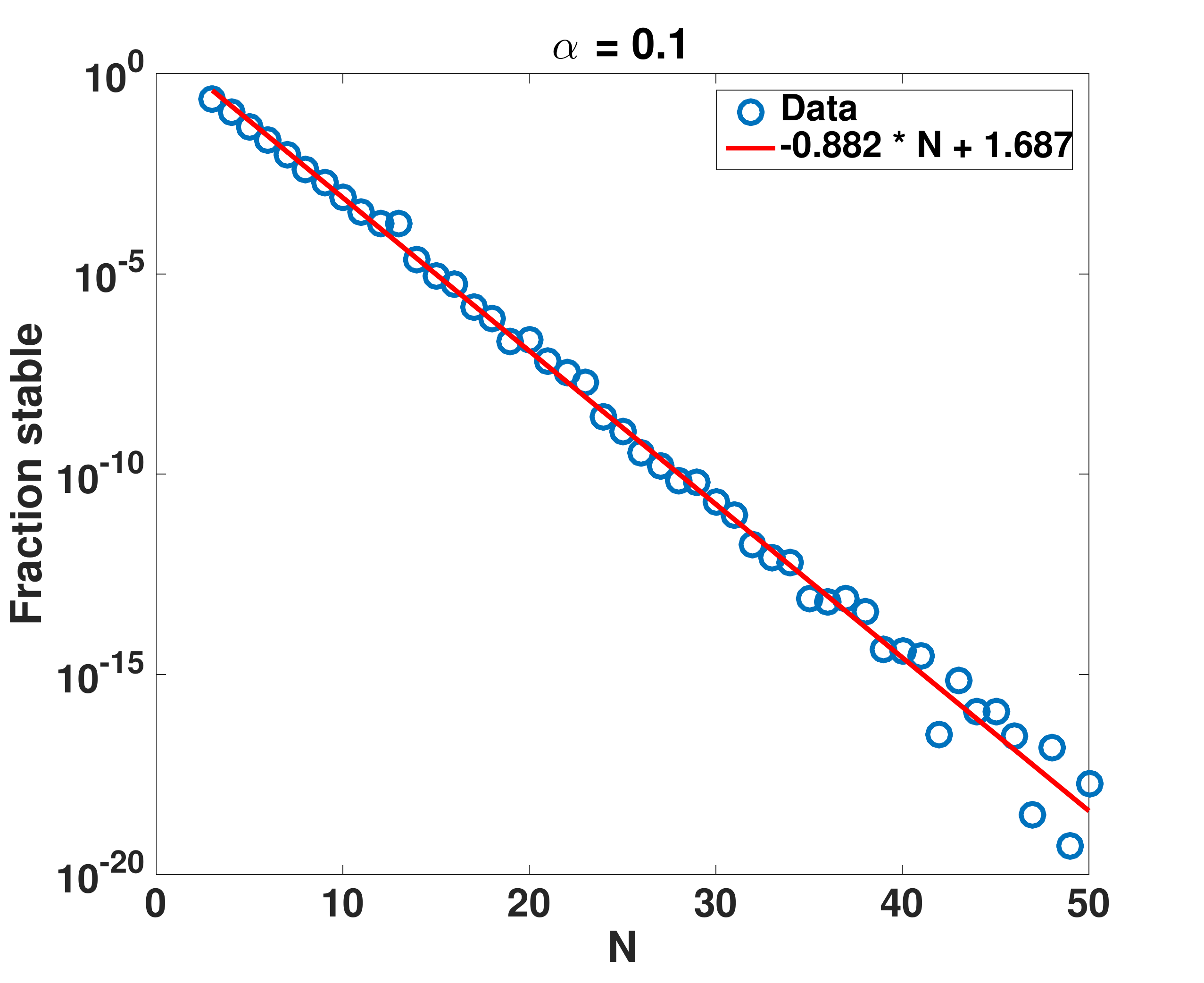}
\end{center}
\caption{Volume of the stable region as a function of $N$ for $\alpha = 0$ and $\alpha = 0.1$.}
\label{fig:Nscale}
\end{figure}

\begin{figure}[th]
\begin{center}
\includegraphics[width=0.8\textwidth]{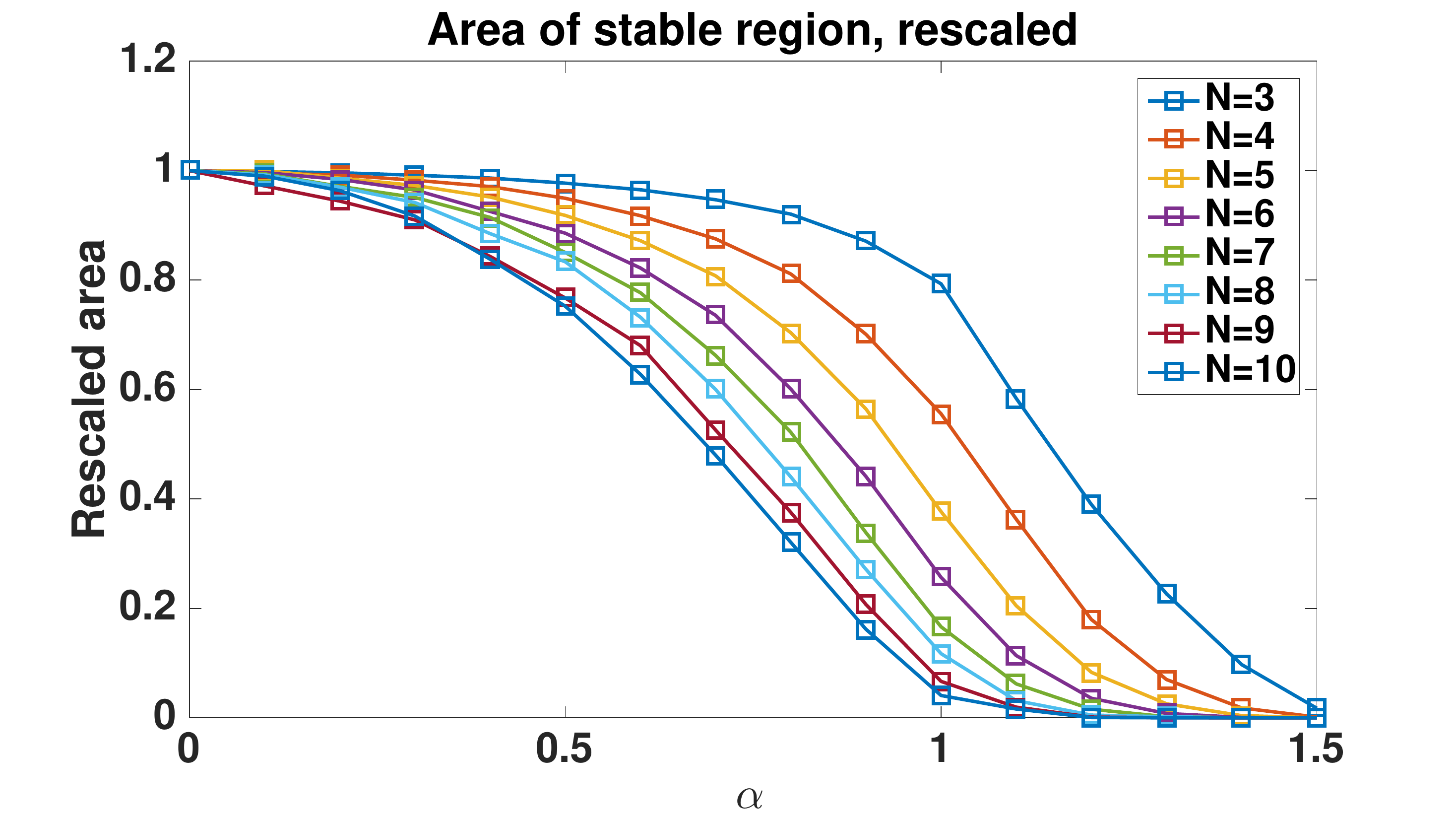}
\end{center}
\caption{(Rescaled) volume of the stable region as a function of $\alpha$ and $N$.  }
\label{fig:AOS}
\end{figure}

In Figure~\ref{fig:Nscale} we plot the volume of the stable set as a function of $N$ for two values of $\alpha$:  $\alpha = 0$ and $\alpha = 0.1$.   As we can see, the volume decays rapidly in $N$, i.e. the volume is $\rho^N$ for some $\rho \in (0,1)$.  The numerics suggest that the decay gives a $\rho$ value somewhere in the $(0.4,0.5)$ range, which we have found by the best least-squares fit.  This justifies the stratification method mentioned above:  for $N=50$ the volume of the stable region is more than fifteen orders of magnitude below unit volume, and to try and capture this volume by direct sampling would be prohibitively expensive.  In all of the numerics done here, we used $100$ strata and sampled each stratum $1000$ times, giving a total of $10^5$ samples for each set of parameters.

In Figure~\ref{fig:AOS} we plot the volume of the stable region as a
function of $\alpha$ and $N$:  each curve represents a fixed value of
$N$, and moving to the right on the curve is an increase in $\alpha$.
Each curve is rescaled so that the standard Kuramoto ($\alpha =0$) for
a given $N$ has unit area, and then we plot the dependence on
$\alpha$.  (Of course, if we did not rescale, then by the results in
Figure~\ref{fig:Nscale}, the curves for large $N$ would be orders of
magnitude smaller and thus not visible on the same plot.  We see that
each of the curves is monotone decreasing  as a function of $\alpha$,
and goes to zero as $\alpha\nearrow\pi/2$.  Note that it falls off
slightly more quickly for larger $N$, but the difference is not that
extreme.

We also note some related ideas appearing in~\cite{Tim}, where the author
has computed bounds on the volume of the stably phase-locked region
for other generalizations of the Kuramoto model --- the model studied there has
similar issues to Kuramoto--Sakaguchi, as the linearization gives a
non-symmetric eigenvalue problem.

\section{Conclusion}\label{sec:outtro}

In this paper we study the stability of phase-locked solutions to the
Kuramoto-Sakaguchi model. We have proved two results, a sufficient
condition for stability and a method for counting the number of
eigenvalues with positive real part modulo two, which gives a
sufficient condition for instability. Numerical evidence in the case
of three oscillators suggests that this count modulo two suffices to
define the asymptotically stable region -- that as the frequency
vector is varied the phase-locked solutions generically transition to
instability via a single real eigenvalue crossing from the left
half-line to the right half-line, and not via a complex conjugate pair
of eigenvalues crossing into the right half-plane. We do not currently
have a proof of this conjecture.

\section{Acknowledgments}

J.C.B. would like to acknowledge support under NSF grant NSF- DMS 1615418.

\noindent T.E.C. would like to acknowledge support from Caterpillar Fellowship Grant at Bradley University.


\end{document}